\documentclass[feqn,openany,11pt]{article}
\usepackage{amsfonts}
\usepackage{amssymb}
\usepackage{xcolor}
\usepackage{amscd}
\usepackage{mathrsfs}
\usepackage{graphicx}
\usepackage{indentfirst,latexsym,amssymb}
\usepackage[bookmarks=false]{hyperref}

\usepackage{makeidx}
\usepackage{amsmath}
\usepackage{amsthm}
\usepackage{epsfig}
\usepackage{graphicx}
\usepackage{titlesec}
\usepackage{bm}
\usepackage{psfrag}
\usepackage{latexsym}
\usepackage{subfigure}
\usepackage{fancyhdr}
\usepackage{esint}
\usepackage[numbers,sort&compress]{natbib}

\usepackage{fancyhdr}

\usepackage{makeidx}
\usepackage{dsfont}
\usepackage{titlesec}
\usepackage{bm}
\usepackage{amscd}
\usepackage{amsmath}
\usepackage{latexsym}
\usepackage{amsfonts}
\usepackage{amssymb}
\usepackage{amsthm}
\usepackage{graphicx}
\usepackage{a4wide,amsfonts,amssymb,amsmath,amsthm, amssymb}
\usepackage{verbatim}
\usepackage{mathrsfs}
\numberwithin{equation}{section}
\allowdisplaybreaks

\makeatletter

\newcommand{\Rmnum}[1]{\expandafter\@slowromancap\romannumeral #1@}

\makeindex          % Switches indexing on (write to \jobname.idx)
\theoremstyle{plain}
\theoremstyle{definition}\newtheorem{theorem}{Theorem}[section]
\theoremstyle{definition}
\theoremstyle{definition}
\theoremstyle{definition}
\theoremstyle{plain}\newtheorem{lemma}[theorem]{Lemma}
\theoremstyle{plain}
\theoremstyle{plain}\newtheorem{proposition}[theorem]{Proposition}
\theoremstyle{definition}\newtheorem{remark}{Remark}[section]

\newcommand{\be}{\begin{equation}}
\newcommand{\ee}{\end{equation}}
\newcommand{\ba}{\begin{aligned}}
\newcommand{\ea}{\end{aligned}}

\newcommand{\ben}{\begin{enumerate}}
\newcommand{\een}{\end{enumerate}}

\def\ef#1 {$(\ref{#1})$}

\begin{document}

%%%%%%%%%%%%%%%%%%%%%%%%%%%%%%%%%%%%%%%%%%%%%%%%%%%%%%%%%%%%%%%%%%%%%%%%%%%%%%%%%%%%%%%%%%%%%%%

%%%%%%%%%%%%%%%%%%%%%%%%%%%%%%%%%%%%%%%%%%%%%%%%%%%%%%%%%%%%%%%%%%%%%%%%%%%%%%%%%%%%%%%%%%%%%

\title{On non-resistive limit of 1D MHD equations with no vacuum at infinity}
\author{Zilai Li$^{1}$, Huaqiao Wang$^2$,  Yulin Ye$^{3}$\thanks{Corresponding author.}}
\date{\today}
\maketitle
\begin{center}
	
	$^{1}$School of Mathematics and Information Science, Henan Polytechnic University, \\Jiaozuo   454000, P.R. China. \\{\tt Email:lizl@hpu.edu.cn}\\
	$^2$College of Mathematics and Statistics, Chongqing University,\\ Chongqing 401331, P.R. China.\\  {\tt Email:wanghuaqiao@cqu.edu.cn}\\
	$^{3}$School of Mathematics and Statistics, Henan University,\\ Kaifeng  475004, P.R. China. \\{\tt Email:ylye@vip.henu.edu.cn}
\end{center}

%-----------------------------------------------------------------------------------
\pagestyle{myheadings} \thispagestyle{plain}\markboth{}{On non-resistive limit of 1D MHD equations with no vacuum at infinity}

\vspace{3mm}

\textbf{Abstract:}
In this paper, we consider the Cauchy problem for the one-dimensional compressible isentropic magnetohydrodynamic (MHD) equations with no vacuum at infinity, but the initial vacuum can be permitted inside the region. By deriving a priori $\nu $ (resistivity coefficient)-independent estimates, we establish the non-resistive limit of the global strong solutions with large initial data. Moreover, as a by-product, the global well-posedness of strong solutions for both the compressible resistive  MHD equations
and non-resistive MHD equations are also established, respectively.


%%%%%%%%%%%%%%%%%%%%%%%%%%%%%%%%%%%%%%%%%%%%%%%%%%%%%%%%%%%%%%%%%%%%%%%%%%%%%%%%%%%%%%%%%%%%%%%%%%%%%%%%%%%%%%%%%%%%%%%%%%%%%
\section{Introduction and Main Results}
Compressible magnetohydrodynamics (MHD) is used to describe the macroscopic behavior of the electrically conducting fluid in a magnetic field.
The application of MHD has a very wide of physical objects from liquid metals to cosmic plasmas. The system of the resistive MHD equations has the form:
\begin{equation}\label{a2}
\begin{cases}
\rho_t+(\rho u)_x=0,\\
(\rho u)_t+(\rho u^2+P(\rho)+\frac{1}{2}b^2)_x=(\mu u_x)_x,\\
b_t+(ub)_x=\nu b_{xx},
\end{cases}
\end{equation}
 in $\mathbb{R}\times [0,\infty)$. Here, $\rho,u, P(\rho)$ and $b$ denote the density, velocity,
pressure and magnetic field, respectively.  $\mu>0$ is the viscosity, the constant $\nu>0$ is the resistivity coefficient
acting as the magnetic diffusion coefficient of the magnetic field. In this paper, we consider the isentropic
compressible MHD equations in which the equation of the state has
the form $$P(\rho)=R\rho^\gamma,\,\;\gamma>1.$$ For simplicity, we set
$R=1$. We focus on the initial condition:
 \begin{equation}\label{a2-1}
 \begin{aligned}
 (\rho,u,b)|_{t=0}=(\rho_{0}(x),u_{0}(x),b_{0}(x))\rightarrow (\bar{\rho},0,\bar{b})\text{ as } |x|\rightarrow+\infty,
 \end{aligned}
 \end{equation}
 where  $\bar{\rho}$ and $\bar{b}$ are both non-zero constants. Note that we can always normalize $\bar{\rho}$ such that $\bar{\rho}\geq 1$.

 However, it is well known that the resistivity coefficient $\nu$ is inversely proportional to the electrical conductivity,
therefore it is more reasonable to ignore the magnetic diffusion which means $\nu=0$, when the conducting fluid considered is of highly conductivity,
for example the ideal conductors. So instead of equations \eqref{a2}, when there is no resistivity, the system reduces to the so called compressible,
isentropic, viscous and non-resistive MHD equations:
\begin{equation}\label{a1}
\begin{cases}
\tilde{\rho}_t+(\tilde{\rho} \tilde{u})_x=0,\\
(\tilde{\rho} \tilde{u})_t+(\tilde{\rho} \tilde{u}^2+P(\tilde{\rho})+\frac{1}{2}\tilde{b}^2)_x=(\mu \tilde{u}_x)_x,\\
\tilde{b}_t+(\tilde{u}\tilde{b})_x=0,
\end{cases}
\end{equation}
in $\mathbb{R}\times [0,+\infty)$, with the following initial condition:
\begin{equation}\label{a1-1}
\begin{aligned}
(\tilde{\rho},\tilde{u},\tilde{b})|_{t=0}=(\tilde{\rho}_0,\tilde{u}_0,\tilde{b}_0)\rightarrow(\bar{\rho},0,\bar{b}) \text{ as }  |x|\rightarrow+\infty.
\end{aligned}
\end{equation}

Because of the tight interaction between the dynamic motion and the magnetic field, the presence of strong nonlinearities,
rich phenomena and mathematical challenges, many physicists and mathematics are attracted to study in this field.
Before stating our main theorems, we briefly recall some previous
known results on compressible MHD equations. Firstly, we begin with
the MHD equations with magnetic diffusion. For one-dimensional case,
Vol'pert-Hudjaev \cite{VH1972} proved the local existence and
uniqueness of strong solutions to the Cauchy problem and
Kawashima-Okada \cite{KO1982} obtained the global smooth solutions
with small initial data. For large initial data and the density containing vacuum,
Ye-Li \cite{YL2019} proved the global existence of strong solutions to the 1D Cauchy problem with vacuum at infinity.
When considering the full MHD equations and
the heat conductivity depends on the temperature $\theta$, Chen-Wang
\cite{CW2002} studied the free boundary value problem and
established the existence, uniqueness and Lipschitz dependence of
strong solutions. Recently, Fan-Huang-Li \cite{FHL2017} obtained the
global strong solutions to the initial boundary value problem to the
planner MHD equations with temperature-dependent heat conductivity.
Later, with the effect of self-gravitation as well as the influence
of radiation on the dynamics at high temperature regimes taken into
account, Zhang-Xie \cite{ZX2008} obtained the global strong
solutions to the initial boundary value problem for the nonlinear
planner MHD equations.  For multi-dimensional MHD equations,
Lv-Shi-Xu \cite{LSX2016} considered the 2-D isentropic MHD equations
and proved the global existence of classical solutions provided that
the initial energy is small, where the decay rates of the
solutions were also obtained. Vol'pert-Hudjaev \cite{VH1972} and
Fan-Yu \cite{FY2009} obtained the local classical solution to the
3-D compressible MHD equations with the initial density is strictly
positive or could contain vacuum, respectively. Hu-Wang
\cite{HW2010} derived the global weak solutions to the 3-D
compressible MHD equations with large initial data. Recently,
Li-Xu-Zhang \cite{LXZ2013} established the global existence of
classical solution of 3-D MHD equations with small energy
but possibly large oscillations. Later, the result was improved by
Hong-Hou-Peng-Zhu \cite{HHPZ2017} just provided
$((\gamma-1)^{\frac{1}{9}}+\nu ^{-\frac{1}{4}})E_0$ is suitably
small. When the resistivity is zero, then the magnetic equation is
reduced from the heat-type equation to the hyperbolic-type equation,
the problem becomes more challenging, hence the results are few.
Kawashima \cite{Kashima1983} obtained the classical solutions to 3-D
MHD equations when the initial data are of small perturbations in
$H^3$-norm and away from vacuum. Xu-Zhang
\cite{XZ2011} proved a blow-up criterion of strong solutions for 3-D
isentropic MHD equations with vacuum. Fan-Hu \cite{FH2015} established the global strong solutions to the initial boundary
value problem of 1-D heat-conducting MHD equations.
With more general heat-conductivity, Zhang-Zhao
\cite{ZZ2017} established the global strong solutions and also
obtained the non-resistivity limits of the solutions in $L^2$-norm.
Jiang-Zhang \cite{JZ2017} obtained the non-resistive limit of the strong solution and the ``magnetic boundary layer" estimates to the
initial boundary value problem of 1-D isentropic MHD equations  as the resistivity $\nu\rightarrow 0$.  Yu
\cite{Y2013} obtained the global existence of strong solutions to
the initial boundary value problem of 1-D isentropic MHD equations. For the Cauchy problem with large initial data and vacuum,
Li-Wang-Ye \cite{LWY2020} established the global well-posedness of strong solutions to the 1D isentropic MHD equations with vacuum at infinity,
that is $\bar{\rho}=0$. However, for the Cauchy problem with no vacuum at infinity, the global well-posedness of strong solutions and the non-resistive limits when the resistivity coefficient $\nu\rightarrow 0$ are still unknown. The goal of this paper is trying to answer these problems.

Now we give some comments on the analysis of this paper. The non-resistive limit of global strong solutions to 1D MHD equations \eqref{a2}-\eqref{a2-1} can be obtained by  global uniform a priori estimates which are independent of resistivity $\nu$.   Thus, to obtain the a priori $\nu $ (resistivity coefficient)-independent  estimates, some of the main new difficulties will be encountered due to the absence of resistivity and  the initial density and initial magnetic which approach non-zero constants at infinity.

It turns out that the key issue in this paper is to derive upper bound for the density, magnetic field and the time-dependent higher norm estimates which are independent of resistivity $\nu$. This is achieved by modifying upper bound estimate for the density developed in \cite{Ye2015} and \cite{YL2019} in the theory of Cauchy problem with vacuum  and initial magnetic approach zero at infinity. However, in comparison with the Cauchy problem with vacuum at infinity in \cite{Ye2015} and \cite{LWY2020}, some new difficulties will be encountered.
The first difficulty lies in no integrability for the density itself just from the elementary energy estimate (see Lemma \ref{lemmac1}), which is required when deriving the upper bound of the density. To overcome this difficulty, we use the technique of mathematical frequency decomposition to divide the momentum $\xi$ into two parts:
$$\xi=\int \rho udx=\int \left(\sqrt{\rho}-\sqrt{\bar{\rho}}\right)\sqrt{\rho}udx+\sqrt{\bar{\rho}}\int \sqrt{\rho}udx=\xi_1+\xi_2.$$
It is crucial to obtain the upper bound of $\xi_1$ and $\xi_2$. For getting $\|\xi_1\|_{L^\infty}$,
we use the technique of mathematical frequency decomposition to get the estimate of $\|\sqrt {\rho}-\sqrt{\bar{\rho}}\|_{L^2}$ by the elementary energy estimates, and then using  H\"{o}lder's inequality, we can obtain the upper bound of $\xi_1$ (see \eqref{c27}). For obtaining $\|\xi_2\|_{L^\infty}$, due to the Sobolev embedding theory, we need some $L^{\tilde{p}}$ integrability of $\xi_2$. However, we can not obtain it just from $\|\xi_{2x}\|_{L^2}$ directly,  because the Poinc\'{a}re type inequality
is no longer valid in the whole space $\mathbb{R}$. To overcome this difficulty, we use the Caffarelli-Kohn-Nirenberg weighted  inequality and the G-N inequality to obtain the upper bound of $\xi_2$ (see \eqref{c28}). It is worth noting that the non-resistive MHD equation \eqref{a1} looks similar to the compressible model for gas
and liquid two-phase fluids. Hence, in some previous results, it is technically assumed that $\tilde{\rho},\tilde{b}\geq 0$ and $0\leq \frac{\tilde{b}}{\tilde{\rho}}<\infty$
in \eqref{a1}, which implies the magnetic field $\tilde{b}$ is bounded provided the density $\tilde{\rho}$ is bounded.
However, this is not physical and realistic in magnetohydrodynamics. Moreover, compared with that for the Cauchy problem with vacuum at infinity in \cite{LWY2020},
since the magnetic field $\tilde{b}\rightarrow \bar{b}$, as $|x|\rightarrow +\infty$, the method of getting the upper bound of magnetic field $b$ in \cite{LWY2020}
can not be used here anymore. So another difficulty is how to get the uniform (independent of $\nu$) upper bound of the magnetic field $b$ without the assumption
similar as that in two-phase fluids.  To overcome this difficulty, we will make full use of the structure of the momentum equation and effective viscous flux (see Lemma \ref{lemmac4} and Lemma \ref{lemmac5}).

\noindent\emph{\bf Notations}: We denote the material derivative of $u$ and
effective viscous flux by
\begin{align*}
\dot{u}\triangleq u_t+uu_x \ \text{and}\
F\triangleq\mu u_x-\left(P(\rho)-P(\bar{\rho})+\frac{b^2-\bar{b}^{2}}{2}\right),
\end{align*}
and define potential energy by
\begin{align*}
\Phi(\rho)=\rho\int_{\bar{\rho}}^{\rho}\frac{P(s)-P(\bar{\rho})}{s^{2}}ds=\frac{1}{\gamma-1}\left(\rho^{\gamma}-\bar{\rho}^\gamma-\gamma \bar{\rho}^{\gamma-1}(\rho-\bar{\rho})\right).
\end{align*}
Sometimes we write $\int_{\mathbb R}f(x) dx$ as $\int f(x) $  for simplicity.

The main results of this paper can be stated as:

As a by-product, the first result is the global existence of strong solutions to 1D non-resistive MHD equations \eqref{a1}-\eqref{a1-1} when the
the initial density and initial magnetic approach non-zero constants at infinity.
 \begin{theorem}\label{theorem1.1}
	Suppose that the initial data $(\tilde{\rho}_0,\tilde{u}_0,\tilde{b}_0)(x)$ satisfies
	\begin{equation}\label{1.1}
	\begin{aligned}
	&\tilde{\rho}_0-\bar{\rho}\in  H^{1}(\mathbb{R}),\ \tilde{b}_0-\bar{b}\in H^1(\mathbb{R}),\ \tilde{u}_0\in H^2(\mathbb{R}),\\
	&\left(\frac{1}{2}\tilde{\rho}_{0} \tilde{u}_{0}^2+\Phi(\tilde{\rho}_{0})+\frac{(\tilde{b}_{0}-\bar{b})^2}{2}\right)|x|^\alpha\in L^1(\mathbb{R})
	\end{aligned}
	\end{equation}
	for $1<\alpha\leq 2$, and the compatibility condition
	
	\begin{equation}\label{1.2}
	\left(\mu\tilde{u}_{0x}-P(\tilde{\rho}_{0})-\frac{1}{2}\tilde{b}_{0}^{2}\right)_{x}=\sqrt{\tilde{\rho}_{0}}\tilde{g}(x),\ \ x\in \mathbb{R}
	\end{equation}
	with some $\tilde{g}$ satisfying $\tilde{g}\in
	L^{2}(\mathbb{R})$. Then for any $T>0$, there exists a unique global strong solution $(\tilde{\rho},\tilde{u},\tilde{b})$
	to the Cauchy problem \eqref{a1} and  \eqref{a1-1} such that
	$$0\leq \tilde{\rho} \leq C,\ (\tilde{\rho}-\bar{\rho},\tilde{b}-\bar{b})\in L^{\infty}(0,T;H^{1}(\mathbb{R})),\ \tilde{\rho}_t\in L^{\infty}(0,T;L^{2}(\mathbb{R})),$$
	$$\tilde{b}_t\in L^{2}(0,T;L^{2}(\mathbb{R})),\ \left(\frac{1}{2}\tilde{\rho} \tilde{u}^2+\Phi(\tilde{\rho})+\frac{(\tilde{b}-\bar{b})^2}{2}\right)|x|^\alpha\in L^{\infty}(0,T;L^{1}(\mathbb{R})),$$
	$$\tilde{u}\in L^{\infty}(0,T;H^{2}(\mathbb{R})),\ \sqrt{\tilde{\rho}}\tilde{u}_t\in L^{\infty}(0,T;L^{2}(\mathbb{R})),\ \tilde{u}_t\in L^2(0,T;H^{1}(\mathbb{R})).$$
	
\end{theorem}

The second result is the global existence and non-resistive limits of strong solutions to 1D  MHD equations \eqref{a2}-\eqref{a2-1} when the
the initial density and initial magnetic approach non-zero constants at infinity.
\begin{theorem}\label{theorem1.2}
	Suppose that the initial data $(\rho_0,u_0,b_0)(x)$ satisfies
	\begin{equation}\label{a3}
	\begin{aligned}
	&\rho_0-\bar{\rho}\in  H^{1}(\mathbb{R}),\ b_0-\bar{b}\in H^1(\mathbb{R}),\ u_0\in H^2(\mathbb{R}),\\
	&\left(\frac{1}{2}\rho_{0} u_{0}^2+\Phi(\rho_{0})+\frac{(b_{0}-\bar{b})^2}{2}\right)|x|^\alpha\in L^1(\mathbb{R})
	\end{aligned}
	\end{equation}
	for $1<\alpha\leq 2$, and the compatibility condition
	
	\begin{equation}\label{a33}
	\left(\mu u_{0x}-P(\rho_{0})-\frac{1}{2}b_{0}^{2}\right)_{x}=\sqrt{\rho_{0}}g(x),\ \ x\in \mathbb{R}
	\end{equation}
	with some $g$ satisfying $g\in
	L^{2}(\mathbb{R})$. Then for each fixed $\nu>0$, there exist a positive constant $C$ and  a unique global strong solution $(\rho,u,b)$
	to the Cauchy problem \eqref{a2}-\eqref{a2-1} such that
	 \begin{equation}\label{a2-2}
	 0\leq \rho (x,t)\leq C,\ \ \ \forall (x,t)\in \mathbb{R}\times [0,T],
	 \end{equation}

	\begin{equation}\label{a2-3}
	\begin{aligned}
	\sup_{0\leq t\leq T}&\left\|\left(\frac{1}{2}{\rho} u^2+\Phi(\rho)+\frac{(b-\bar{b})^2}{2}\right)(1+|x|^\alpha)\right\|_{L^1}(t)\\
	&\quad+\int_0^T \left(\mu\left\| u_x(1+|x|^{\frac{\alpha}{2}})\right\|_{L^2}^2+\nu\left\| b_x(1+|x|^{\frac{\alpha}{2}})\right\|_{L^2}^2\right)dt\leq C,
	\end{aligned}
	\end{equation}
	and
	\begin{equation}\label{a2-4}
	\begin{aligned}
	\sup_{0\leq t\leq T}&\left(\|u\|_{L^2}+\|u_{xx}\|_{L^2}+\|u_x\|_{L^2}+\|b_x\|_{L^2}+\|\rho_x\|_{L^2}+\|\sqrt{\rho}\dot{u}\|_{L^2}\right)(t)\\
	&+\int_0^T\left(\|u_{xt}\|_{L^2}^2+\nu\|b_{xx}\|_{L^2}^2\right)dt\leq C.
	\end{aligned}
	\end{equation}
	Moreover, as $\nu\rightarrow0$, we have
	\begin{equation}\label{a2-5}
	\begin{cases}
	(\rho,u,b)\rightarrow(\tilde{\rho},\tilde{u},\tilde{b})\text{ strongly in } L^\infty(0,T;L^2),\\
	\nu b_x\rightarrow0,\ \ u\rightarrow\tilde{u},\ \  u_x\rightarrow\tilde{u}_x\text{ strongly in } L^2(0,T;L^2),\\
	\end{cases}
	\end{equation}
	and
	\begin{equation}\label{a2-6}
		\begin{aligned}
	\sup_{0\leq t\leq T}&\left(\|\rho-\tilde{\rho}\|_{L^2}^2+\|u-\tilde{u}\|_{L^2}^2+\|b-\tilde{b}\|_{L^2}^2\right)+\int_0^T\mu\|(u-\tilde{u})_x\|_{L^2}^2dt \leq C\nu,
	\end{aligned}
	\end{equation}
	where $C$ is a positive constant independent of $\nu$.
\end{theorem}
\begin{remark}
In Theorems \ref{theorem1.1} and \ref{theorem1.2}, we do not need the artificial assumption similarly as that in two-phase fluids. Moreover, if
ignoring the magnetic field, then MHD system reduces to the compressible Navier-Stokes equations. So, Theorems \ref{theorem1.1} and \ref{theorem1.2}
can be seen as an extension of that in \cite{Ye2015}.
\end{remark}
\begin{remark}
In Theorem \ref{theorem1.2}, we give that the global strong solution of resistive MHD equation \eqref{a2}-\eqref{a2-1} converges
to that of non-resistive MHD equation \eqref{a1}-\eqref{a1-1} in $L^2$-norm as $\nu \rightarrow 0$, moreover, the convergence rates are also justified.
\end{remark}

The rest of the paper is organized as follows. In Section \ref{s2},
we recall some preliminary lemmas which will be used later. Section
\ref{s3} is devoted to establishing global $\nu$-independent estimates for \eqref{a2} and \eqref{a2-1},
which will be used to justify the non-resistive limit. Sections \ref{s4} and \ref{s5}
are devoted to proving  Theorem \ref{theorem1.1} and Theorem \ref{theorem1.2}, respectively.

%%%%%%%%%%%%%%%%%%%%%%%%%%%%%%%%%%%%%%%%%%%%%%%%%%%%%%%%%%%%%%%%%%%%%%%%%%%%%%%%%%%%%%%%%%%%%%%%%%%%%%%%%%%%%%%%%%%%%%%%%%%%
\section{Preliminaries}\label{s2}
In this section, we will recall some known facts and elementary inequalities that will be used frequently later.

 The following well-known  inequality will be used frequently later.
\begin{lemma}
[\bf Gagliardo-Nirenberg inequality \cite{Gagliardo,Nirenberg}]\label{lemmab1} For any $f\in W^{1,m}({\mathbb{R}})\cap L^r({\mathbb{R}})$, there exists some generic
 constant $C>0$ which may depend  on $q,r$ such that
\begin{equation}\label{b1-1}
\|f\|_{L^q}\le C \|f\|_{L^r}^{1-\theta}\|\nabla
f\|_{L^m}^{\theta} ,
\end{equation}
where $\theta=(\frac{1}{r}-\frac{1}{q})(\frac{1}{r}-\frac{1}{m}+1)^{-1}$, if $m=1$, then $q\in[r,\infty)$, if $m>1$, then $q\in[r,\infty]$.
\end{lemma}

The following Caffarelli-Kohn-Nirenberg weighted inequality is the key to deal with the Cauchy problem in this paper.
\begin{lemma}[\bf Caffarelli-Kohn-Nirenberg weighted inequality \cite{CKN1984}]\label{lemmab2}
(1) $\forall h\in C_{0}^{\infty}(R)$, it holds that
\begin{equation}\label{b1}
\|{|x|^{\kappa}h}\|_{r}\leq C\|{|x|^{\alpha}|\partial_{x} h|}\|_{p}^{\theta}\|{|x|^{\beta}h}\|_{q}^{1-\theta}
\end{equation}
where $1\leq p,q<\infty,0<r<\infty,0\leq\theta\leq1,\frac{1}{p}+\alpha>0,
\frac{1}{q}+\beta>0,\frac{1}{r}+\kappa>0$ and satisfy
\begin{equation}\label{b2}
\frac{1}{r}+\kappa=\theta\left(\frac{1}{p}+\alpha-1\right)+(1-\theta)
\left(\frac{1}{q}+\beta\right),
\end{equation}
and\\
$$\kappa=\theta\sigma+(1-\theta)\beta$$\\
with $0\leq\alpha-\sigma \ if\ \theta>0\ \ and \ 0\leq\alpha-\sigma\leq1\ if\ \theta>0\ and\ \frac{1}{p}+\alpha-1=\frac{1}{r}+\kappa$.\\
\end{lemma}
\begin{proof}The proof of (1) can be found in \cite{CKN1984}. Here, we omit the details.
\end{proof}
By direct calculation, the potential energy $\Phi(\rho,\bar{\rho})$ has the following properties:
\begin{lemma}\label{Lem2.4}
	Observing the function of the potential energy $\Phi(\rho,\bar{\rho})$, we will easily find the following properties for  positive constants  $c_{1}, c_{2}, C_1, C_2$:

	(1) if $0\leq \rho\leq 2\bar{\rho}$,
	$c_{1}(\rho-\bar{\rho})^{2}\leq \Phi(\rho,\bar{\rho})\leq c_{2}(\rho-\bar{\rho})^{2};$
	
    (2) if $\rho >2\bar{\rho},$
	$\rho^{\gamma}-\bar{\rho}^{\gamma}\leq C_1(\rho-\bar{\rho})^{\gamma}\leq C_2\Phi(\rho,\bar{\rho}).$
\end{lemma}
%%%%%%%%%%%%%%%%%%%%%%%%%%%%%%%%%%%%%%%%%%%%%%%%%%%%%%%%%%%%%%%%%%%%%%%%%%%
\section{Global $\nu$-independent estimates for \eqref{a2} and \eqref{a2-1}}\label{s3}
The main purpose of this section is to derive the global $\nu$-independent a priori estimates of the solutions $(\rho, u, b)$ to
the system \eqref{a2} and \eqref{a2-1}, which is used to justify the non-resistive limit. Hence, in this section,
we assume $(\rho, u, b)$ is a smooth solution of the system \eqref{a2} and \eqref{a2-1} on $\mathbb{R}\times [0,T]$ with $0<T<\infty$.
For the sake of simplicity, we denote by $C$ the generic positive constant, which may depend on $\gamma,\mu,T$, but is independent of $\nu$.

First of all, we can prove the following elementary energy estimates.
\begin{lemma}\label{lemmac1}
	Let $(\rho,u,b)$  be a smooth solution of \eqref{a2} and \eqref{a2-1}.
	Then for any $T>0$, it holds that
	\begin{equation}\label{lemmac2-1}
	\sup_{0\leq t\leq T}\left\|\frac{1}{2}\rho
	u^{2}+\Phi(\rho)+\frac{(b-\bar{b})^2}{2}\right\|_{L^1(\mathbb{R})}+
	\int_{0}^{T}\left(\mu
	\|u_{x}\|_{L^2}^{2}+\nu\| b_x\|_{L^2}^2\right)dt\leq
	C.
	\end{equation}
	
\end{lemma}
\begin{proof}
	By the defination of the potential energy $ \Phi$, and using $\eqref{a2}_{1}$, we deduce
	\begin{equation}\label{l1-1}
	\Phi_{t}+(u\Phi)_{x}+(\rho^{\gamma}-\bar{\rho}^{\gamma})u_{x}=0.
	\end{equation}
	Adding the equation $\eqref{a2}_{2}$ multiplied by $u$, $\eqref{a2}_3$ multiplied by $b$, into \eqref{l1-1}, then
	integrating the resulting equation over $\mathbb{R}\times [0,T]$ with respect to the variables $x$ and $t$, we have
	\begin{equation*}\label{c2}
	\begin{aligned}
&\int \left(\frac{1}{2}\rho u^{2}+\Phi(\rho)+\frac{(b-\bar{b})^2}{2}\right)dx+\int_{0}^{T}\left(\mu\|u_{x}\|_{L^2}^{2}+\nu\|b_x\|_{L^2}^2\right)dt\\
&\leq\int\left( \frac{1}{2}\rho_{0}u^{2}+\Phi(\rho_{0})+\frac{(b_{0}-\bar{b})^2}{2}\right)dx\\
&\leq\|\rho_{0}\|_{L^{\infty}}\|u_{0}\|^{2}_{L^{2}}+C(\|\rho_{0}-\bar{\rho}\|_{L^{2}}+\|b_{0}-\bar{b}\|^{2}_{L^{2}})\\
&\leq C(\|\rho_{0}-\bar{\rho}\|_{H^{1}}+\bar{\rho})(\|u_{0}\|^{2}_{L^{2}}+1)+C\|b_{0}-\bar{b}\|^{2}_{L^{2}}\\
&\leq C.
	\end{aligned}
	\end{equation*}
	Then the proof of Lemma \ref{lemmac1} is completed.
\end{proof}

To obtain the upper bound of the density $\rho$, we need the following weighted energy estimates.
\begin{lemma}\label{lemmac2}
		Let $(\rho,u,b)$  be a smooth solution of \eqref{a2} and \eqref{a2-1}.
	Then for any $T>0$ and some index $1<\alpha \leq 2$, it holds that
	\begin{equation}\label{c2}
	\begin{aligned}
\sup_{0\leq t\leq T}&\left\|\left(\frac{1}{2}\rho u^2+\Phi(\rho)+\frac{(b-\bar{b})^2}{2}\right)|x|^\alpha\right\|_{L^1(\mathbb{R})} +\int_0^T\left(\mu\|u_x|x|^{\frac{\alpha}{2}}\|_{L^2}^2+\nu\|b_x|x|^{\frac{\alpha}{2}}\|_{L^2}^2\right)dt\leq C.
	\end{aligned}
	\end{equation}
	\end{lemma}
\begin{proof}
Multiplying the equation $\eqref{a2}_2$ by $u|x|^\alpha$ and integrating the resulting equation over $\mathbb{R}$ with respect to $x$, we have
\begin{equation}\label{c3}
\begin{aligned}
\frac{1}{2}\frac{d}{dt}&\int \rho u^2 |x|^\alpha+\mu \int u_x^2 |x|^\alpha\\
&=\frac{1}{2}\int \rho u^3\alpha |x|^{\alpha-2}x-\mu\int \alpha|x|^{\alpha-2}xuu_x-\int \left(P(\rho)+\frac{b^2}{2}\right)_x u|x|^\alpha.\\
\end{aligned}
\end{equation}
It follows from the integration by parts that
\begin{equation}\label{c4}
\begin{aligned}
-&\int \left(P(\rho)+\frac{b^2}{2}\right)_x u|x|^\alpha\\&=-\int \left(\left(P(\rho)-P(\bar{\rho})\right)_x + bb_x\right)u|x|^\alpha\\
&=-\int \left(\left(P(\rho)-P(\bar{\rho})\right)_x + (b-\bar{b})(b-\bar{b})_x+\bar{b}(b-\bar{b})_x\right)u|x|^\alpha\\
&=\int \left(P(\rho)-P(\bar{\rho}) + \frac{(b-\bar{b})^2}{2}+\bar{b}(b-\bar{b})\right)\left(u_x|x|^\alpha+u\alpha|x|^{\alpha-2}x\right).\\
\end{aligned}
\end{equation}
Then, combining \eqref{c3} and \eqref{c4} gives
\begin{equation}\label{c5}
\begin{aligned}
\frac{1}{2}\frac{d}{dt}&\int \rho u^2 |x|^\alpha+\mu \int u_x^2 |x|^\alpha
=\frac{1}{2}\int \rho u^3\alpha |x|^{\alpha-2}x-\mu\int \alpha|x|^{\alpha-2}xuu_x\\
&+\int \left(P(\rho)-P(\bar{\rho}) + \frac{(b-\bar{b})^2}{2}+\bar{b}(b-\bar{b})\right)\left(u_x|x|^\alpha+u\alpha|x|^{\alpha-2}x\right).
\end{aligned}
\end{equation}
To deal with the last term on the right-hand side of \eqref{c5}, first, multiplying \eqref{l1-1} by $|x|^\alpha$ and integrating over $\mathbb{R}$ with respect to $x$, yields
\begin{equation}\label{c7}
\begin{aligned}
\frac{d}{dt}\int \Phi(\rho )|x|^\alpha -\int u\Phi(\rho)\alpha|x|^{\alpha-2}x+\int (\rho^\gamma-\bar{\rho}^\gamma)|x|^\alpha u_x=0,
\end{aligned}
\end{equation}
 and then multiplying the equation $\eqref{a2}_3$ by $(b-\bar{b})|x|^\alpha$ and integrating over $\mathbb{R}$ with respect to $x$, we have
\begin{equation*}\label{c8}
\begin{aligned}
\frac{1}{2}\frac{d}{dt}\int (b-\bar{b})^2|x|^\alpha+&\frac{1}{2}\int \left((b-\bar{b})^2\right)_xu|x|^\alpha+\int (b-\bar{b})^2 u_x|x|^\alpha+\bar{b}\int (b-\bar{b})u_x|x|^\alpha\\
&=\nu \int b_{xx}(b-\bar{b})|x|^\alpha,
\end{aligned}
\end{equation*}
which together with the integration by parts implies that
\begin{equation}\label{c9}
\begin{aligned}
\frac{1}{2}\frac{d}{dt}\int(b-\bar{b})^2|x|^\alpha&+\nu\int b_x^2|x|^\alpha=-\int \frac{(b-\bar{b})^2}{2}(u_x|x|^\alpha-u\alpha|x|^{\alpha-2}x)\\
&-\bar{b}\int (b-\bar{b})u_x|x|^\alpha-\nu\int b_x(b-\bar{b})\alpha|x|^{\alpha-2}x.
\end{aligned}
\end{equation}
Now, putting \eqref{c7} and \eqref{c9} into \eqref{c5}, we have
\begin{equation}\label{c10}
\begin{aligned}
\frac{d}{dt}&\int \left(\frac{1}{2}\rho u^2+\Phi(\rho)+\frac{(b-\bar{b})^2}{2}\right)|x|^\alpha dx+\mu\int u_x^2|x|^\alpha+\nu\int b_x^2 |x|^\alpha\\
&=\int \left(\frac{1}{2}\rho u^2+\Phi(\rho)+(b-\bar{b})^2\right)u\alpha|x|^{\alpha-2}x+\int (\rho^\gamma-\bar{\rho}^\gamma)u \alpha |x|^{\alpha-2}x\\
&\ \ \ \  -\mu\int \alpha |x|^{\alpha-2}x uu_x+\bar{b}\int (b-\bar{b})u \alpha |x|^{\alpha-2}x-\nu\int (b-\bar{b})b_x  \alpha |x|^{\alpha-2}x\\
&=:I_1+I_2+I_3+I_4+I_5.
\end{aligned}
\end{equation}
Next, we estimate the terms $I_1$-$I_5$ as follows:
\begin{equation}\label{c11}
\begin{aligned}
I_1&\leq \int \left(\left|\frac{1}{2}\rho u^2+\Phi(\rho)+\frac{(b-\bar{b})^2}{2}\right|\right)|u||x|^{\alpha-1}\\
&\leq \int\left(\left|\frac{1}{2}\rho u^2+\Phi(\rho)+\frac{(b-\bar{b})^2}{2}\right||x|^{\alpha}\right)^{\frac{\alpha-1}{\alpha}}\left(\left|\frac{1}{2}\rho u^2+\Phi(\rho)+\frac{(b-\bar{b})^2}{2}\right|\right)^{\frac{1}{\alpha}}|u|\\
&\leq C\|u\|_{L^\infty}\left\|\left(\frac{1}{2}\rho u^2+\Phi(\rho)+\frac{(b-\bar{b})^2}{2}\right)|x|^{\alpha}\right\|_{L^1}^{1-\frac{1}{\alpha}}\left\|\frac{1}{2}\rho u^2+\Phi(\rho)+\frac{(b-\bar{b})^2}{2}\right\|_{L^1}^{\frac{1}{\alpha}}\\
&\leq C(1+\|u_x\|_{L^2})\left\|\left(\frac{1}{2}\rho u^2+\Phi(\rho)+\frac{(b-\bar{b})^2}{2}\right)|x|^{\alpha}\right\|_{L^1}^{1-\frac{1}{\alpha}}\\
&\leq C(1+\|u_x\|_{L^2})\left(\left\|\left(\frac{1}{2}\rho u^2+\Phi(\rho)+\frac{(b-\bar{b})^2}{2}\right)|x|^{\alpha}\right\|_{L^1}+1\right),
\end{aligned}
\end{equation}
where $\alpha>1$ and  we have used H\"{o}lder's inequality, Lemma \ref{lemmac1} and the following facts:
\begin{equation*}\label{c12}
\begin{aligned}
\bar{\rho}\int u^2&=\int (\bar{\rho}-\rho+\rho)u^2\\
&\leq \int \left((\bar{\rho}-\rho)|_{\{0\leq\rho\leq 2\bar{\rho}\}}+(\bar{\rho}-\rho)|_{\{\rho> 2\bar{\rho}\}}\right)u^2+\int \rho u^2\\
&\leq C\left(\|(\rho-\bar{\rho})|_{\{0\leq\rho\leq 2\bar{\rho}\}}\|_{L^2}\|u\|_{L^4}^2+\|(\rho-\bar{\rho})|_{\{\rho> 2\bar{\rho}\}}\|_{L^\gamma}\|u\|_{L^{\frac{2\gamma}{\gamma-1}}}^2+1\right)\\
&\leq C\left[\|\Phi(\rho)\|_{L^1}^{\frac{1}{2}}\left(\|u\|_{L^2}^{\frac{3}{4}}\|u_x\|_{L^2}^{\frac{1}{4}}\right)^{2}+\|\Phi( \rho)\|_{L^1}^{\frac{1}{\gamma}}\left(\|u\|_{L^2}^{1-\frac{1}{2\gamma}}\|u_x\|_{L^2}^{\frac{1}{2\gamma}}\right)^2+1\right]\\
&\leq \varepsilon \|u\|_{L^2}^2+ C\left(1+\|u_x\|_{L^2}^2\right),
\end{aligned}
\end{equation*}
where we have used the properties of the potential energy $\Phi$ in Lemma \ref{Lem2.4}. This together with G-N inequality implies that
\begin{equation}\label{c12-1}
\|u\|_{L^2}^2\leq  C(1+\|u_x\|_{L^2}^2),
\end{equation}
and
\begin{equation}\label{c13}
\|u\|_{L^\infty}\leq C\|u\|_{L^2}^{\frac{1}{2}}\|u_x\|_{L^2}^{\frac{1}{2}}\leq C(1+\|u_x\|_{L^2}).
\end{equation}
Applying the property of $\Phi(\rho)$, H\"{o}lder's inequality, the C-K-N weighted inequality \eqref{b2}, Lemma \ref{lemmac1}, Young's inequality and \eqref{c13}, we obtain
\begin{equation}\label{c14}
\begin{aligned}
I_2&=\int \left((\rho^\gamma-\bar{\rho}^\gamma)|_{\{0\leq \rho\leq 2\bar{\rho}\}}+(\rho^\gamma-\bar{\rho}^\gamma)|_{\{\rho> 2\bar{\rho}\}}\right)u\alpha |x|^{\alpha-2}x\\
&\leq C\int \left(|\rho-\bar{\rho}||_{\{0\leq \rho\leq 2\bar{\rho}\}}+(\rho-\bar{\rho})^\gamma|_{\{\rho> 2\bar{\rho}\}}\right)|u||x|^{\alpha-1}\\
&\leq C\int \left(\Phi^{\frac{1}{2}}(\rho)|_{\{0\leq \rho\leq 2\bar{\rho}\}}+\Phi(\rho)|_{\{\rho> 2\bar{\rho}\}}\right)|u||x|^{\alpha-1}\\
&\leq C\int \left(\Phi^{\frac{1}{2}}(\rho)|_{\{0\leq \rho\leq 2\bar{\rho}\}}|x|^{\frac{\alpha}{2}}\right)|u||x|^{\frac{\alpha}{2}-1}+C\int \left(\Phi(\rho)|_{\{\rho> 2\bar{\rho}\}}|x|^{\alpha}\right)^{\frac{\alpha-1}{\alpha}}\Phi^{\frac{1}{\alpha}}(\rho)|u|\\
&\leq C\left(\|\Phi(\rho)|x|^{\alpha}\|_{L^1}^{\frac{1}{2}}\left\||x|^{\frac{\alpha}{2}-1}u\right\|_{L^2}+\|\Phi(\rho)|x|^\alpha\|_{L^1}^{1-\frac{1}{\alpha}}\|\Phi(\rho)\|_{L^1}^{\frac{1}{\alpha}}\|u\|_{L^\infty}\right)\\
&\leq +C(1+\|u_x\|_{L^2}^2)(1+\|\Phi(\rho)|x|^\alpha\|_{L^1}),
\end{aligned}
\end{equation}
where we have used the fact:
\begin{equation}\label{cc}
	\begin{aligned}
		\||x|^{\frac{\alpha}{2}-1}u\|_{L^2}&\leq C\|u\|_{L^2}^{\frac{\alpha}{2}}\|u_x\|_{L^2}^{1-\frac{\alpha}{2}}\\
		&\leq C\left(\|u\|_{L^2}+\|u_x\|_{L^2}\right)\\
		&\leq C(1+\|u_x\|_{L^2}),
	\end{aligned}
\end{equation}
here, the index $1<\alpha\leq 2$.

Similarly, using the C-K-N weighted inequality \eqref{b2}, we have
\begin{equation}\label{c15}
\begin{aligned}
I_3&\leq C\||x|^{\frac{\alpha}{2}}u_x\|_{L^2}\||x|^{\frac{\alpha}{2}-1}u\|_{L^2}\\
&\leq \varepsilon\||x|^{\frac{\alpha}{2}}u_x\|_{L^2}^2 +C\left(1+\|u_x\|_{L^2}^2\right).
\end{aligned}
\end{equation}
For the term $I_4$, it follows from the H\"{o}lder inequality and C-K-N weighted inequality that
\begin{equation}\label{c19}
\begin{aligned}
I_4&\leq C\left\|(b-\bar{b})|x|^{\frac{\alpha}{2}}\right\|_{L^2}\left\||x|^{\frac{\alpha}{2}-1}u\right\|_{L^2}\\
	&\leq C\left(\left\|(b-\bar{b})|x|^{\frac{\alpha}{2}}\right\|_{L^2}^2+\|u_x\|_{L^2}^2+1\right).
\end{aligned}
\end{equation}
Similarly, using the C-K-N weighted inequality \eqref{b2}, one has
\begin{equation}\label{c20}
\begin{aligned}
I_5&\leq C\nu \||x|^{\frac{\alpha}{2}}b_x\|_{L^2}\|(b-\bar{b})|x|^{\frac{\alpha}{2}-1}\|_{L^2}\\
&\leq C\nu \||x|^{\frac{\alpha}{2}}b_x\|_{L^2}\left(\|b-\bar{b}\|_{L^2}^{\frac{\alpha}{2}}\|b_x\|_{L^2}^{1-\frac{\alpha}{2}}\right)\\
&\leq \varepsilon\nu  \||x|^{\frac{\alpha}{2}}b_x\|_{L^2}^2+C\nu\left(1+\|b_x\|_{L^2}^2\right),
\end{aligned}
\end{equation}
where $1<\alpha\leq 2$.

Then, substituting \eqref{c11}, \eqref{c14}, \eqref{c15}, \eqref{c19} and \eqref{c20} into \eqref{c10}, choosing $\varepsilon>0$ small enough, one deduce
\begin{equation*}
\begin{aligned}
\frac{d}{dt}&\int \left(\frac{1}{2}\rho u^2+\Phi(\rho)+\frac{(b-\bar{b})^2}{2}\right)|x|^\alpha+\mu\int |x|^\alpha u_x^2+\nu\int |x|^\alpha b_x^2\\
&\leq C(1+\|u_x\|_{L^2}^2)\left\|\left(\frac{1}{2}\rho u^2+\Phi(\rho)+\frac{(b-\bar{b})^2}{2}\right)|x|^\alpha\right\|_{L^1}+C(1+\|u_x\|_{L^2}^2+\nu\|b_x\|_{L^2}^2).
\end{aligned}
\end{equation*}
This together with Gronwall's inequality, \eqref{a3} and Lemma \ref{lemmac1} gives
$$\int \left(\frac{1}{2}\rho u^2+\Phi(\rho)+\frac{(b-\bar{b})^2}{2}\right)|x|^\alpha+\mu\int_0^T\int |x|^\alpha u_x^2+\nu\int_0^T\int |x|^\alpha b_x^2\leq C(T).$$
Then, the proof of Lemma \ref{lemmac2} is completed.
\end{proof}

The upper bound of the density $\rho$ can be shown in the similar manner as that in \cite{Ye2015}. However, for completeness of the paper, we give the details here.
\begin{lemma}\label{lemmac3}
		Let $(\rho,u,b)$  be a smooth solution of \eqref{a2} and \eqref{a2-1}.
	Then for any $T>0$, it holds that
	\begin{equation*}
	0\leq \rho(x,t)\leq C, \ \forall (x,t)\in \mathbb{R}\times[0,T],
	\end{equation*}
	and
	$$\sup_{0\leq t\leq T}\|\rho-\bar{\rho}\|_{L^2(\mathbb{R})}\leq C.$$
	\end{lemma}
\begin{proof}
Let $\xi=\int_{-\infty}^x \rho udy$, then the momentum equation $\eqref{a2}_2$ can be rewritten as
\begin{equation*}
\xi_{tx}+\left(\rho u^2+P(\rho)-P(\bar{\rho})+\frac{b^2-\bar{b}^2}{2}\right)_x=(\mu u_x)_x.
\end{equation*}	
Integrating the above equality with respect to $x$ over $(-\infty,x)$ yields that
\begin{equation}\label{c21}
\xi_t+\rho u^2 +P(\rho)-P(\bar{\rho})+\frac{b^2-\bar{b}^2}{2}=\mu u_x,
\end{equation}
which together with $\eqref{a2}_1$ gives
\begin{equation}\label{c22}
\xi_t+\rho u^2+P(\rho)-P(\bar{\rho})+\frac{b^2-\bar{b}^2}{2}+\mu\frac{\rho_t+u\rho_x}{\rho}=0.
\end{equation}
Next, we define the particle trajectory $X(x,t)$ as follows:
\begin{equation}\label{c23}
\begin{cases}
\frac{dX(x,t)}{dt}=u(X(x,t),t),\\
X(x,0)=x,
\end{cases}
\end{equation}
which implies
\begin{equation*}\label{c24}
\begin{aligned}
\frac{d\xi}{dt}=\xi_t+u\xi_x=\xi_t+\rho u^2.
\end{aligned}
\end{equation*}
Then combining the above equality and \eqref{c22}, we infer that
\begin{equation*}
\frac{d}{dt}(\xi+\mu\ln \rho)\left(X(x,t),t\right) +\left(P(\rho)+\frac{b^2}{2}\right)\left(X(x,t),t\right)-\left(P(\bar{\rho})+\frac{\bar{b}}{2}\right)=0,
\end{equation*}
which together with $\left(P(\rho)+\frac{b^2}{2}\right)\left(X(x,t),t\right)\geq 0$ yields
\begin{equation}\label{c25}
\frac{d}{dt}(\xi+\mu\ln \rho)\left(X(x,t),t\right)\leq\left( P(\bar{\rho})+\frac{\bar{b}^2}{2}\right)\left(X(x,t),t\right)\leq C.
\end{equation}
Thus, integrating it over $[0,T]$ with respect to $t$, we have
$$(\xi +\mu\ln \rho)\left(X(x,t),t\right) \leq (\xi+\mu\ln \rho)(x,0)+C.$$
By direct calculation, we obtain
\begin{equation}\label{c26}
\begin{aligned}
\ln \rho& \leq\frac{1}{\mu} (\xi_0+\mu\ln \rho_0+C-\xi)\\
&\leq\frac{1}{\mu} \left(  \xi_0+\mu\ln \rho_0+C-\int_{-\infty}^x \rho u dy\right)\\
&\leq\frac{1}{\mu} \left(  \xi_0+\mu\ln \rho_0+C-\int_{-\infty}^x \sqrt{\rho}u(\sqrt{\rho}-\sqrt{\bar{\rho}}) dy-\sqrt{\bar{\rho}}\int_{-\infty}^x\sqrt{\rho} u dy\right)\\
&\leq\frac{1}{\mu} ( \xi_0+\ln \rho_0+C-\xi_1-\xi_2).
	\end{aligned}
\end{equation}
Firstly, it follows from Lemma \ref{Lem2.4} and Lemma \ref{lemmac1} that $\|\xi_1\|_{L^\infty}$ can be estimated as
\begin{equation}\label{c27}
\begin{aligned}
|\xi_1|&=|\int_{-\infty}^x \sqrt{\rho}u(\sqrt{\rho}-\sqrt{\bar{\rho}}) dy|\\
&\leq \|\sqrt{\rho}u\|_{L^2}\left(\|(\sqrt{\rho}-\sqrt{\bar{\rho}})|_{\{0\leq \rho \leq 2\bar{\rho}\}}\|_{L^2}+\|(\sqrt{\rho}-\sqrt{\bar{\rho}})|_{\{\rho> 2\bar{\rho}\}}\|_{L^2}\right)\\
&\leq C\|\sqrt{\rho}u\|_{L^2}\left(\|(\rho-\bar{\rho})|_{\{0\leq \rho \leq 2\bar{\rho}\}}\|_{L^2}+\|\sqrt{\rho -\bar{\rho}}|_{\{\rho >2\bar{\rho}\}}\|_{L^2}\right)\\
&\leq C\|\sqrt{\rho}u\|_{L^2}\left(\|(\rho-\bar{\rho})|_{\{0\leq \rho \leq 2\bar{\rho}\}}\|_{L^2}+\|(\rho -\bar{\rho})_{\{\rho >2\bar{\rho}\}}\|_{L^1}^{\frac{1}{2}}\right)\\
&\leq C\|\sqrt{\rho}u\|_{L^2}\left(\|(\rho-\bar{\rho})|_{\{0\leq \rho \leq 2\bar{\rho}\}}\|_{L^2}+\|(\rho -\bar{\rho})^\gamma|_{\{\rho >2\bar{\rho}\}}\|_{L^1}^{\frac{1}{2}}\right)\\
&\leq C\|\sqrt{\rho}u\|_{L^2}\left(\|\Phi(\rho)\|_{L^1}^{\frac{1}{2}}+\|\Phi(\rho)\|_{L^1}^{\frac{1}{2}}\right)\leq C.
\end{aligned}
\end{equation}
Secondly, using the Gagliardo-Nirenberg inequality, the Caffarelli-Kohn-Nirenberg weighted inequality, Lemma \ref{lemmac1} and Lemma \ref{lemmac2}, we can bound  $\|\xi_2\|_{L^{\infty}}$ as
\begin{equation}\label{c28}
\begin{aligned}
\|\xi_2\|_{L^{\infty}}&\leq C\|\xi_2\|_{L^{\tilde{p}}}^{\frac{\tilde{p}}{\tilde{p}+2}}\|\xi_{2x}\|_{L^2}^{\frac{2}{\tilde{p}+2}}\\
&\leq C(\|\xi_{2x}\|_{L^{2}}^{\eta}\||x|^\kappa \xi_2\|_{L^{q}}^{1-\eta})^{\frac{\tilde{p}}{\tilde{p}+2}}\|\xi_{2x}\|_{L^2}^{\frac{2}{\tilde{p}+2}}\\
&\leq C(\|\xi_{2x}\|_{L^{2}}^{\eta}\||x|^{\frac{\alpha}{2}}\xi_{2x}\|_{L^2}^{1-\eta})^{\frac{\tilde{p}}{\tilde{p}+2}}\|\xi_{2x}\|
_{L^2}^{\frac{2}{\tilde{p}+2}}\\
&\leq C\|\xi_{2x}\|_{L^{2}}^{1-(1-\eta)\frac{\tilde{p}}{\tilde{p}+2}}\||x|^{\frac{\alpha}{2}}\xi_{2x}\|_{L^2}^{(1-\eta)\frac{\tilde{p}}{\tilde{p}+2}}\\
&= C\|\xi_{2x}\|_{L^{2}}^{1-\frac{1}{\alpha}}\||x|^{\frac{\alpha}{2}}\xi_{2x}\|_{L^2}^{\frac{1}{\alpha}}\\
&\leq C(\|\sqrt{\rho}u\|^{1-\frac{1}{\alpha}}_{L^{2}}\||x|^{\frac{\alpha}{2}}\sqrt{\rho}u\|^{\frac{1}{\alpha}}_{L^2})\leq C.
\end{aligned}
\end{equation}
Here the indexes $1\leq \tilde{p}<\infty,\ q>1, \ \eta\in (0,1)$ and satisfy
\begin{equation*}
\begin{aligned}
 &\frac{1}{\tilde{p}}=(\frac{1}{2}-1)\eta+\left(\frac{1}{q}+\kappa\right)(1-\eta),\\
  &\frac{1}{q}+\kappa=\frac{1}{2}+\frac{\alpha}{2}-1>0,
\end{aligned}
\end{equation*}
which gives
\begin{equation}
\begin{aligned}
\alpha>1,\ \eta=\frac{\tilde{p}(\alpha-1)-2}{\alpha \tilde{p}}>0\Rightarrow \tilde{p}>\frac{2}{\alpha-1}.
\end{aligned}
\end{equation}
Similarly as that for \eqref{c27} and \eqref{c28}, we can obtain
\begin{equation}\label{c26}
|\xi_0|\leq C.
\end{equation}
Then, substituting \eqref{c27}, \eqref{c28} and \eqref{c26} into \eqref{c25}, we have
$$\ln\rho\leq C,$$
which gives
$$\rho \leq C.$$
It follows from Lemma \ref{Lem2.4} and Lemma \ref{lemmac1} that
\begin{equation}\label{c28-1}
\begin{aligned}
\|\rho-\bar{\rho}\|_{L^{2}}&=\left\|(\rho-\bar{\rho})|_{\{\rho\leq2\bar{\rho}\}}^2\right\|_{L^1}^{\frac{1}{2}}+\left\|(\rho-\bar{\rho})|_{\{\rho\geq2\bar{\rho}\}}\right\|_{L^2}\\
&\leq\left\|\Phi(\rho)\right\|_{L^1}^{\frac{1}{2}}+\left\|(\rho-\bar{\rho})|_{\{\rho\geq2\bar{\rho}\}}\right\|_{L^2}\\
&\leq C+\left\|(\rho-\bar{\rho})|_{\{\rho\geq2\bar{\rho}\}}\right\|_{L^2}.
\end{aligned}
\end{equation}
To deal with the last term on the right-hand side of \eqref{c28-1}, we discuss it in the following two cases:\\
\underline{Case 1}: if $1<\gamma\leq 2$, by Lemma \ref{Lem2.4}, Lemma\ref{lemmac1} and the fact that $\rho\le C$, it holds
\begin{equation}\label{c28-3}
\|(\rho -\bar{\rho})|_{\{\rho\geq2\bar{\rho}\}}\|_{L^2(\mathbb{R})}\leq \|\rho -\bar{\rho}\|^{\frac{\gamma}{2}}_{L^\gamma(\mathbb{R})}\|\rho -\bar{\rho}\|^{1-\frac{\gamma}{2}}_{L^\infty(\mathbb{R})}\leq \|\Phi(\rho)\|^{\frac{1}{2}}_{L^1(\mathbb{R})}\|\rho -\bar{\rho}\|^{1-\frac{\gamma}{2}}_{L^\infty(\mathbb{R})}\leq C(T).
\end{equation}
\underline{Case 2}: if $\gamma>2$, using the fact that
$\rho-\bar{\rho} >\bar{\rho}\Rightarrow\frac{1}{(\rho-\bar{\rho})^{\gamma-2}}<\frac{1}{\bar{\rho}^{\gamma-2}}$ and  Lemma \ref{lemmac1}, we have
\begin{equation}\label{c28-4}
\begin{aligned}
&\|(\rho -\bar{\rho})|_{\{\rho\geq2\bar{\rho}\}}\|_{L^2(\mathbb{R})}\\
&\leq\left\|(\rho-\bar{\rho})^2|_{\{\rho\geq2\bar{\rho}\}}\right\|_{L^1}^{\frac{1}{2}}\\
&\leq\left\|(\rho-\bar{\rho})^\gamma|_{\{\rho\geq2\bar{\rho}\}}\frac{1}{(\rho-\bar{\rho})^{\gamma-2}|_{\{\rho\geq2\bar{\rho}\}}}\right\|_{L^1}^{\frac{1}{2}}\\
&\leq \left\|(\rho-\bar{\rho})^\gamma|_{\{\rho\geq2\bar{\rho}\}}\frac{1}{\bar{\rho}^{\gamma-2}}\right\|_{L^1}^{\frac{1}{2}}\\
&\leq \frac{1}{\bar{\rho}^{\gamma-2}}\left\|\Phi(\rho)\right\|_{L^1}^{\frac{1}{2}}\leq C.\\
\end{aligned}
\end{equation}
Thus, combining \eqref{c28-3}, \eqref{c28-4} with  \eqref{c28-1}, we can obtain  $\|\rho-\bar{\rho}\|_{L^{2}}\leq C.$
This completes the proof of Lemma \ref{lemmac3}.
 \end{proof}

\begin{lemma}\label{lemmac4}
Let $(\rho,u,b)$  be a smooth solution of \eqref{a2}-\eqref{a2-1}.
Then for any $T>0$, it holds that
\begin{equation*}
\begin{aligned}
\sup_{0\leq t\leq T}\Big(\mu \|u_x\|_{L^2}^2&+\nu \|b_x\|_{L^2}^2+\|b-\bar{b}\|_{L^4}^4\Big)+\int_0^T\| b-\bar{b}\|_{L^6}^6+\mu\nu\|(b-\bar{b})b_x\|_{L^2}^2 dt\\
&+\int_0^T \nu^2 \|b_{xx}\|_{L^2}^2 dt +\int_0^T\|\sqrt{ \rho }\dot{u}\|_{L^2}^2dt \leq C
\end{aligned}
\end{equation*}	
and
\begin{equation*}
 \sup_{0\leq t\leq T}(\|u\|_{L^2}+\|u\|_{L^{\infty}})\leq C.
\end{equation*}	
\end{lemma}
\begin{proof}The proof of Lemma \ref{lemmac4} will be divided into four steps.
	
	\underline{Step 1.}
Multiplying the equation $\eqref{a2}_2$ by $\dot{u}$ and integrating the resulting equation over $\mathbb{R}$ with respect to $x$ yields
\begin{equation}\label{c29}
\begin{aligned}
\frac{\mu}{2}&\frac{d}{dt}\int u_x^2+\int \rho \dot{u}^2\\
&=-\mu \int u_x(uu_x)_x-\int P(\rho)_x(u_t+uu_x)-\int \left(\frac{b^2}{2}\right)_x(u_t+uu_x)\\
&=:J_1+J_2+J_3.
\end{aligned}
\end{equation}
Firstly, by integration by parts, we find
\begin{equation}\label{c30}
\begin{aligned}
J_1&=-\mu \int u_x^3-\mu \int u\left(\frac{u_x^2}{2}\right)_x\\
&=-\mu \int u_x^3+\mu \int \frac{u_x^2}{2}\cdot u_x\\
&=-\frac{\mu}{2}\int u_x^3.
\end{aligned}
\end{equation}
Similarly, we have
\begin{equation}\label{c31}
\begin{aligned}
	J_2&=-\int \left(P(\rho)-P(\bar{\rho})\right)_x(u_t+uu_x)\\
	&=\frac{d}{dt}\int \left(P(\rho)-P(\bar{\rho})\right)u_x-\int\left[ \left(P(\rho)-P(\bar{\rho})\right)_t+ \left(P(\rho)-P(\bar{\rho})\right)_x u\right]u_x\\
	&=\frac{d}{dt}\int \left(P(\rho)-P(\bar{\rho})\right)u_x+\gamma\int \rho^\gamma u_x^2,
\end{aligned}
\end{equation}
and
\begin{equation}\label{c32}
\begin{aligned}
J_3&=-\int b(b-\bar{b})_x(u_t+uu_x)\\
&=-\int \left(\frac{(b-\bar{b})^2}{2}\right)_x(u_t+uu_x)-\bar{b}\int (b-\bar{b})_x(u_t+uu_x)\\
&=\frac{d}{dt}\int \left(\frac{(b-\bar{b})^2}{2}+\bar{b}(b-\bar{b})\right)u_x-\int \left(\left(\frac{(b-\bar{b})^2}{2}\right)_t+\left(\frac{(b-\bar{b})^2}{2}\right)_x u\right)u_x\\
&\quad -\bar{b}\int \left((b-\bar{b})_t+u(b-\bar{b})_x\right)u_x\\
&=\frac{d}{dt}\int \left(\frac{(b-\bar{b})^2}{2}+\bar{b}(b-\bar{b})\right)u_x-\int \left((b-\bar{b})+\bar{b}\right)(\nu b_{xx}-bu_x)u_x,
\end{aligned}
\end{equation}
where we have used the following facts:
$$(P(\rho)-P(\bar{\rho}))_t+u(P(\rho)-P(\bar{\rho}))_x+\gamma \rho ^\gamma u_x=0,$$
$$\left(\frac{(b-\bar{b})^2}{2}\right)_t+u\left(\frac{(b-\bar{b})^2}{2}\right)_x+b(b-\bar{b})u_x=\nu b_{xx}(b-\bar{b}),$$
and
$$(b-\bar{b})_t+u(b-\bar{b})_x+b u_x=\nu b_{xx}.$$
Substituting \eqref{c30}-\eqref{c32} into \eqref{c29}, using  H\"{o}lder's and Cauchy-Schwarz's inequalities, we have
\begin{equation}\label{c33}
\begin{aligned}
\frac{d}{dt}&\int \frac{\mu}{2}u_x^2-\left(P(\rho)-P(\bar\rho)\right)u_x-\frac{(b-\bar{b})^2}{2}u_x-\bar{b}(b-\bar{b})u_x dx+\int \rho \dot{u}^2\\
&=-\frac{\mu}{2}\int u_x^3-\gamma\int \rho^\gamma u_x^2 -\nu \int b_{xx}(b-\bar{b})u_x+\int (b-\bar{b})^2 u_x^2\\
&\quad+2\bar{b}\int (b-\bar{b})u_x^2-\bar{b}\int \nu b_{xx}u_x+\bar{b}^2\int u_x^2\\
&\leq C\big(\|u_x\|_{L^3}^3+\|u_x\|_{L^2}^2+\nu \|b_{xx}\|_{L^2}\|b-\bar{b}\|_{L^6}\|u_x\|_{L^3}+\|u_x\|_{L^3}^2\|b-\bar{b}\|_{L^6}^2\\
&\quad +\|u_x\|_{L^3}^2\|b-\bar{b}\|_{L^3}+\nu \|b_{xx}\|_{L^2}\|u_x\|_{L^2}\big)\\
&\leq \varepsilon \nu^2 \|b_{xx}\|_{L^2}^2+C\left(\|u_x\|_{L^3}^3+\|b-\bar{b}\|_{L^6}^6 +(\|b-\bar{b}\|_{L^2}^{\frac{1}{2}}\|b-\bar{b}\|_{L^6}^{\frac{1}{2}})^3+\|u_x\|_{L^2}^2\right)\\
&\leq \varepsilon \nu^2 \|b_{xx}\|_{L^2}^2+C\left(\|u_x\|_{L^3}^3+\|b-\bar{b}\|_{L^6}^6+\|u_x\|_{L^2}^2+1\right).
\end{aligned}
\end{equation}
Now, we estimate the term $\|u_x\|_{L^3}^3$ on the right-hand side of the equation \eqref{c33}.
Employing the effective viscous flux, momentum equation, Lemma \ref{lemmac1} and Lemma \ref{lemmac3}, we have
\begin{equation}\label{c35}
\begin{aligned}
\|F\|_{L^3}&\leq\|F\|^{\frac{5}{6}}_{L^2}\|F_x\|^{\frac{1}{6}}_{L^2}\\
&\leq C\left\|\left[\mu u_x-\left(P(\rho)-P(\bar{\rho})\right)-\left(\frac{(b-\bar{b})^2}{2}+\bar{b}(b-\bar{b})\right) \right]\right\|^{\frac{5}{6}}_{L^2}\|\rho \dot{u}\|^{\frac{1}{6}}_{L^2}\\
&\leq C\left(\|u_x\|_{L^2}+\|\rho-\bar{\rho}\|_{L^2}+\|b-\bar{b}\|_{L^4}^2+\|b-\bar{b}\|_{L^2}\right)^{\frac{5}{6}}\|\rho \dot{u}\|^{\frac{1}{6}}_{L^2}\\
&\leq C\left(\|u_x\|_{L^2}+\|b-\bar{b}\|_{L^2}^{\frac{1}{2}}\|b-\bar{b}\|_{L^6}^{\frac{3}{2}}+1\right)^{\frac{5}{6}}\|\rho \dot{u}\|^{\frac{1}{6}}_{L^2}\\
&\leq C\left(\|u_x\|_{L^2}+\|b-\bar{b}\|_{L^6}^{\frac{3}{2}}+1\right)^{\frac{5}{6}}\|\sqrt{\rho} \dot{u}\|^{\frac{1}{6}}_{L^2}.
\end{aligned}
\end{equation}
Moreover, we can obtain
\begin{equation}\label{c38}
\begin{aligned}
\|u_x\|_{L^3}&=\left\|\frac{F+P(\rho)-P(\bar{\rho})+\frac{b^2-\bar{b}^2}{2}}{\mu}\right\|_{L^3}\\
&\leq C\left(\|F\|_{L^3}+\|\rho-\bar{\rho}\|_{L^3}+\left\|\frac{(b-\bar{b})^2}{2}+\bar{b}(b-\bar{b})\right\|_{L^3}\right)\\
&\leq C\left(\|F\|_{L^3}+\|\rho-\bar{\rho}\|_{L^2}^{\frac{2}{3}}\|\rho-\bar{\rho}\|_{L^\infty}^{\frac{1}{3}}+\|b-\bar{b}\|_{L^6}^2+\|b-\bar{b}\|_{L^3}\right)\\
&\leq C\left(\left(\|u_x\|_{L^2}+\|b-\bar{b}\|_{L^6}^{\frac{3}{2}}+1\right)^{\frac{5}{6}}\|\sqrt{\rho} \dot{u}\|^{\frac{1}{6}}_{L^2}+\|b-\bar{b}\|_{L^6}^2+1\right),
\end{aligned}
\end{equation}
which implies
\begin{equation}\label{c39}
\begin{aligned}
\|u_x\|_{L^3}^3&\leq C\left(\left(\left(\|u_x\|_{L^2}+\|b-\bar{b}\|_{L^6}^{\frac{3}{2}}+1\right)^{\frac{5}{6}}\|\sqrt{\rho}\dot{u}\|_{L^2}^{\frac{1}{6}}\right)^3+\|b-\bar{b}\|_{L^6}^6+1\right)\\
&\leq \varepsilon\|\sqrt{\rho}\dot{u}\|_{L^2}^2+C\left(\|u_x\|_{L^2}^4+\|b-\bar{b}\|_{L^6}^6+1\right).
\end{aligned}
\end{equation}
Then, substituting \eqref{c39} into \eqref{c33} and choosing $\varepsilon$ sufficiently small, we have
\begin{equation}\label{c40}
\begin{aligned}
\frac{d}{dt}&\int \frac{\mu}{2}u_x^2-\left(P(\rho)-P(\bar{\rho})\right)u_x-\frac{(b-\bar{b})^2}{2}u_x-\bar{b}(b-\bar{b})u_x +\int \rho \dot{u}^2\\
&\leq \varepsilon\nu^2 \|b_{xx}\|_{L^2}^2+\varepsilon\|\sqrt{\rho}\dot{u}\|_{L^2}^2+C\left(\|u_x\|_{L^2}^4+\|b-\bar{b}\|_{L^6}^6+1\right).
\end{aligned}
\end{equation}

\underline{Step 2.} To control the term $\|b-\bar{b}\|_{L^6}^6$ on the right-hand side of \eqref{c40}, we rewrite the magnetic field equation as
$$(b-\bar{b})_t+u(b-\bar{b})_x+(b-\bar{b})u_x+\bar{b}u_x=\nu b_{xx}.$$
Then multiplying  the above equation by $(b-\bar{b})^3$ and integrating it over $\mathbb{R}$ with respect to $x$, we have
\begin{equation*}\label{c41}
\begin{aligned}
\frac{1}{4}&\frac{d}{dt}\int (b-\bar{b})^4+3\nu\int(b-\bar{b})^2 b_x^2=-\int\left(\frac{3}{4} (b-\bar{b})^4 +\bar{b} (b-\bar{b})^3\right) u_x \\
&=-\int\left(\frac{3}{4} (b-\bar{b})^4 +\bar{b} (b-\bar{b})^3\right)\frac{F+P(\rho)-P(\bar{\rho})+\frac{b^2-\bar{b}^2}{2}}{\mu}\\
&=-\int\left(\frac{3}{4} (b-\bar{b})^4 +\bar{b} (b-\bar{b})^3\right)\frac{F+P(\rho)-P(\bar{\rho})+\frac{(b-\bar{b})^2}{2}+\bar{b}(b-\bar{b})}{\mu}\\
&=-\frac{3}{4\mu}\int (b-\bar{b})^4\left(F+P(\rho)-P(\bar{\rho})\right)-\frac{\bar{b}}{\mu}\int (b-\bar{b})^3\left(F+P(\rho)-P(\bar{\rho})\right)\\
&\quad-\frac{3}{8\mu}\int (b-\bar{b})^6-\frac{5\bar{b}}{4\mu}\int (b-\bar{b})^5-\frac{\bar{b}^2}{\mu}\int (b-\bar{b})^4,
\end{aligned}
\end{equation*}
which gives
\begin{equation}\label{c42}
\begin{aligned}
&\frac{1}{4}\frac{d}{dt}\int(b-\bar{b})^4+\frac{3}{8\mu}\int (b-\bar{b})^6+\frac{\bar{b}^2}{\mu}\int (b-\bar{b})^4+3\nu \int (b-\bar{b})^2b_x^2\\
&=-\frac{3}{4\mu}\int (b-\bar{b})^4 \left(F+P(\rho)-P(\bar{\rho})\right)-\frac{\bar{b}}{\mu}\int (b-\bar{b})^3\left(F+P(\rho)-P(\bar{\rho})\right)-\frac{5\bar{b}}{4\mu}\int (b-\bar{b})^5\\
&=:K_1+K_2+K_3,
\end{aligned}
\end{equation}
Next, we estimate $K_1-K_3$ term by term. Using \eqref{c35}, Lemmas \ref{lemmac3} and Young's inequality, we have
\begin{equation}\label{c43}
\begin{aligned}
K_1&\leq C\|b-\bar{b}\|_{L^6}^4\left(\|F\|_{L^3}+\|P(\rho)-P(\bar{\rho})\|_{L^3}\right)\\
&\leq C\|b-\bar{b}\|_{L^6}^4\left(\left(\|u_x\|_{L^2}+\|b-\bar{b}\|_{L^6}^{\frac{3}{2}}+1\right)^{\frac{5}{6}}\|\sqrt{\rho}\dot{u}\|_{L^2}^{\frac{1}{6}}+\|\rho-\bar{\rho}\|_{L^3}\right)\\
&\leq \varepsilon\|b-\bar{b}\|_{L^6}^6+\varepsilon \|\sqrt{\rho}\dot{u}\|_{L^2}^2+C(1+\|u_x\|_{L^2}^4).
\end{aligned}
\end{equation}
Similarly, we can get
\begin{equation}\label{c44}
\begin{aligned}
K_2&\leq \|b-\bar{b}\|_{L^6}^3\left(\|F\|_{L^2}+\|P(\rho)-P(\bar{\rho})\|_{L^2}\right)\\
&\leq \|b-\bar{b}\|_{L^6}^3\left(\|u_x\|_{L^2}+\|\rho-\bar{\rho}\|_{L^2}+\|b-\bar{b}\|^{2}_{L^4}+\|b-\bar{b}\|_{L^2}\right)\\
&\leq \|b-\bar{b}\|_{L^6}^3\left(\|u_x\|_{L^2}+\|\rho-\bar{\rho}\|_{L^2}+\|b-\bar{b}\|^{\frac{1}{2}}_{L^2}\|b-\bar{b}\|^{\frac{3}{2}}_{L^6}+\|b-\bar{b}\|_{L^2}\right)\\
&\leq \varepsilon\|b-\bar{b}\|_{L^6}^6+C\left(\|u_x\|_{L^2}^4+1\right),
\end{aligned}
\end{equation}
and
\begin{equation}\label{c45}
\begin{aligned}
K_3&\leq C\|b-\bar{b}\|_{L^5}^5\leq C\|b-\bar{b}\|_{L^2}^{\frac{1}{2}}\|b-\bar{b}\|_{L^6}^{\frac{9}{2}}\leq \varepsilon\|b-\bar{b}\|_{L^6}^6+C.
\end{aligned}
\end{equation}
Then, substituting \eqref{c43}-\eqref{c45} into \eqref{c42} and choosing $\varepsilon>0$ small enough yields
\begin{equation}\label{c46}
\begin{aligned}
\frac{1}{4}\frac{d}{dt}&\int(b-\bar{b})^4+\frac{3}{8\mu}\int (b-\bar{b})^6+\frac{\bar{b}^2}{\mu}\int (b-\bar{b})^4+3\nu \int (b-\bar{b})^2b_x^2\\
&\leq \varepsilon\|\sqrt{\rho}\dot{u}\|_{L^2}^2+C\left(1+\|u_x\|_{L^2}^4\right).
\end{aligned}
\end{equation}

\underline{Step 3.}  To control the term $\nu^2 \|b_{xx}\|_{L^2}^2$ on the right hand-side of \eqref{c40}, we multiply the equation $\eqref{a2}_3$
by $\nu b_{xx}$ and integrate it over $\mathbb{R}$ to get
\begin{equation}\label{c47}
\begin{aligned}
\frac{\nu}{2}\frac{d}{dt}\int b_x^2+\nu^2\int b_{xx}^2&=\nu \int b_{xx}ub_x+\nu \int b_{xx}(b-\bar{b})u_x+\bar{b}\nu\int b_{xx}u_x\\
&=-\frac{\nu}{2}\int b_x^2 u_x+\nu \int b_{xx}(b-\bar{b})u_x+\bar{b}\nu\int b_{xx}u_x\\
&=:H_1+H_2+H_3.
\end{aligned}
\end{equation}
For the term $H_1$, by the effective viscous flux, it shows that
\begin{equation}\label{c48}
\begin{aligned}
H_1&=-\frac{\nu}{2}\int b_x^2\frac{F+P(\rho)-P(\bar{\rho})+\frac{b^2-\bar{b}^2}{2}}{\mu}\\
&\leq -\frac{\nu}{2\mu}\int b_x^2 F+\frac{\nu}{2\mu}\int b_x^2\left(P(\bar{\rho})+\frac{\bar{b}^2}{2}\right)\\
&\leq C\left(1+\|F\|_{L^\infty}\right)\nu \|b_x\|_{L^2}^2\\
&\leq C\left(1+\|\sqrt{\rho}\dot{u}\|_{L^2}+\|u_x\|_{L^2}+\|b-\bar{b}\|_{L^6}^{\frac{3}{2}}\right)\nu \|b_x\|_{L^2}^2\\
&\leq \varepsilon \|\sqrt{\rho}\dot{u}\|_{L^2}^2+ \|b-\bar{b}\|_{L^6}^{6}+C\left(1+\|u_x\|_{L^2}^2+\nu \|b_x\|_{L^2}^2\right)\nu \|b_x\|_{L^2}^2,
\end{aligned}
\end{equation}
where we have used the following inequality:
\begin{equation}\label{c49}
\begin{aligned}
\|F\|_{L^\infty}&\leq C\|F\|_{L^2}^{\frac{1}{2}}\|F_x\|_{L^2}^{\frac{1}{2}}\\
&\leq C\left(\|u_x\|_{L^2}+\|\rho-\bar{\rho}\|_{L^2}+\|b-\bar{b}\|_{L^4}+\|b-\bar{b}\|_{L^2}\right)^{\frac{1}{2}}\|\sqrt{\rho}\dot{u}\|_{L^2}^{\frac{1}{2}}\\
	&\leq C\left(1+\|\sqrt{\rho}\dot{u}\|_{L^2}+\|u_x\|_{L^2}+\|b-\bar{b}\|_{L^6}^{\frac{3}{2}}\right).
\end{aligned}
\end{equation}
For the terms $H_2$ and $H_3$, using H\"{o}lder's inequality, Cauchy's inequality and \eqref{c39}, we have
\begin{equation}\label{c50}
\begin{aligned}
H_2&\leq \nu \|b_{xx}\|_{L^2}\|b-\bar{b}\|_{L^6}\|u_x\|_{L^3}\\
&\leq \varepsilon \nu^2\|b_{xx}\|_{L^2}^2+C\left(\|b-\bar{b}\|_{L^6}^6+\|u_x\|_{L^3}^3\right)\\
&\leq \varepsilon \nu^2\|b_{xx}\|_{L^2}^2+\varepsilon \|\sqrt{\rho}\dot{u}\|_{L^2}^2+C\left(1+\|u_x\|_{L^2}^4+\|b-\bar{b}\|_{L^6}^6\right),
\end{aligned}
\end{equation}
and
\begin{equation}\label{c51-1}
\begin{aligned}
H_3&\leq C\nu \|b_{xx}\|_{L^2}\|u_x\|_{L^2}\leq \varepsilon\nu^2\|b_{xx}\|_{L^2}^2+C\|u_x\|_{L^2}^2.
\end{aligned}
\end{equation}
Thus, substituting \eqref{c48}, \eqref{c50} and \eqref{c51-1} into \eqref{c47} and choosing $\varepsilon>0$ small enough, we have
\begin{equation}\label{c52}
\begin{aligned}
\frac{\nu}{2}\frac{d}{dt}\int b_x^2+\nu^2\int b_{xx}^2&\leq 2\varepsilon\|\sqrt{\rho}\dot{u}\|_{L^2}^2+C\left(1+\|u_x\|_{L^2}^2+\nu\|b_x\|_{L^2}^2\right)\nu \|b_x\|_{L^2}^2\\
&\quad +C\left(1+\|u_x\|_{L^2}^4+\|b-\bar{b}\|_{L^6}^6\right).
\end{aligned}
\end{equation}

\underline{Step 4.} Adding the equation \eqref{c46} multiplied by $\frac{8\mu}{3}(2C+1)$ and \eqref{c40} into \eqref{c52}, we have
\begin{equation}\label{c53}
\begin{aligned}
\frac{d}{dt}&\int\left( \frac{\mu}{2}u_x^2+\frac{\nu}{2}b_x^2+\frac{2\mu(2C+1)}{3}(b-\bar{b})^4-\psi(\rho,u,b)\right)+\int \rho \dot{u}^2+\nu^2\int b_{xx}^2\\
&\quad +\int\left( (b-\bar{b})^6+\frac{8\bar{b}^2(2C+1)}{3} (b-\bar{b})^4+8\mu\nu(2C+1) (b-\bar{b})^2b_x^2\right)\\
&\leq C+C\left(1+\|u_x\|_{L^2}^2+\nu\|b_x\|_{L^2}^2\right)\left(\|u_x\|_{L^2}^2+\nu\|b_x\|_{L^2}^2\right),
\end{aligned}
\end{equation}
where $\psi(\rho,u,b)=\left(P(\rho)-P(\bar{\rho})\right)u_x+\frac{(b-\bar{b})^2}{2}u_x+\bar{b}(b-\bar{b})u_x.$
Integrating \eqref{c53} over $[0,T]$ with respect to $t$, we obtain
\begin{equation}\label{c54}
\begin{aligned}
\int&\left(\mu u_x^2+\nu b_x^2+(b-\bar{b})^4\right)+\int_0^T\int \left(\rho \dot{u}^2+(b-\bar{b})^6+\mu\nu(b-\bar{b})^2b_x^2+\nu^2 b_{xx}^2\right)\\
&\leq C+C\int_0^T\left(1+\|u_x\|_{L^2}^2+\nu\|b_x\|_{L^2}^2\right)\left(\|u_x\|_{L^2}^2+\nu\|b_x\|_{L^2}^2\right)\\
&\quad+\int \psi(\rho,u,b)-\int \psi(\rho_0,u_0,b_0).
\end{aligned}
\end{equation}
By Lemma \ref{lemmac1}, Lemma \ref{lemmac3} and Cauchy-Schwarz's inequality, we can obtain
\begin{equation*}\label{c51}
\begin{aligned}
&\int \psi(\rho,u,b)-\int \psi(\rho_0,u_0,b_0)\\
&\leq C\left(\|P(\rho)-P(\bar{\rho})\|_{L^2}+\|b-\bar{b}\|_{L^4}^2+\|b-\bar{b}\|_{L^2}\right)\|u_x\|_{L^2}+C\\
&\leq C\left(\|\rho -\bar{\rho}\|_{L^2}+\|b-\bar{b}\|_{L^4}^2+1\right)\|u_x\|_{L^2}+C\\
&\leq \varepsilon\mu\|u_x\|_{L^2}^2+C(1+\|b-\bar{b}\|_{L^4}^4),
\end{aligned}
\end{equation*}
which together with  \eqref{c54} gives (choosing $\varepsilon>0$ small enough)
\begin{equation}\label{c56}
\begin{aligned}
\int&\left(\mu u_x^2+\nu b_x^2+(b-\bar{b})^4\right)+\int_0^T\int \left(\rho \dot{u}^2+(b-\bar{b})^6+\mu\nu(b-\bar{b})^2b_x^2+\nu^2 b_{xx}^2\right)\\
&\leq C(1+\|b-\bar{b}\|_{L^4}^4)+C\int_0^T\left(1+\|u_x\|_{L^2}^2+\nu\|b_x\|_{L^2}^2\right)\left(\|u_x\|_{L^2}^2+\nu\|b_x\|_{L^2}^2\right).
\end{aligned}
\end{equation}

Now, integrating \eqref{c46} over $[0,T]$ with respect to $t$, and then adding the resulting inequality multiplied by $4C$ into \eqref{c56} show that
\begin{equation*}\label{c57}
\begin{aligned}
\int&\left(\mu u_x^2+\nu b_x^2+(b-\bar{b})^4\right)+\int_0^T\int \left(\rho \dot{u}^2+(b-\bar{b})^6+\mu\nu(b-\bar{b})^2b_x^2+\nu^2 b_{xx}^2\right)\\
&\leq C+C\int_0^T\left(1+\|u_x\|_{L^2}^2+\nu\|b_x\|_{L^2}^2\right)\left(\|u_x\|_{L^2}^2+\nu\|b_x\|_{L^2}^2\right),
\end{aligned}
\end{equation*}
which together with Gronwall's inequality and Lemma \ref{lemmac1} yields
\begin{equation*}
\begin{aligned}
\int \mu u_x^2&+\nu b_x^2+(b-\bar{b})^4dx+\int_0^T\int (b-\bar{b})^6+\mu\nu(b-\bar{b})^2b_x^2 dxdt\\
&+\int_0^T\int \nu^2 b_{xx}^2 dxdt +\int_0^T\int \rho \dot{u}^2 dxdt \leq C(T).
\end{aligned}
\end{equation*}	
Thus, it follows from  \eqref{c12-1} and \eqref{c13} that
\begin{equation*}
\|u\|_{L^2}^2\leq  C(1+\|u_x\|_{L^2}^2)\leq  C,
\end{equation*}
and
\begin{equation*}
\|u\|_{L^\infty}\leq C\|u\|_{L^2}^{\frac{1}{2}}\|u_x\|_{L^2}^{\frac{1}{2}}\leq C(1+\|u_x\|_{L^2})\leq  C.
\end{equation*}
Then, we complete the proof of Lemma \ref{lemmac4}.
\end{proof}

To get the uniform upper bound of the magnetic field $b$, we need to re-estimate the $\|b_x\|_{L^2}$ independent of $\nu$ as follows.
\begin{lemma}\label{lemmac5}
	Let $(\rho,u,b)$  be a smooth solution of \eqref{a2}-\eqref{a2-1}.
	Then for any $T>0$, it holds that
	$$\sup_{0\leq t\leq T}\left(\|b\|_{L^\infty(\mathbb{R})}+ \|\rho_x\|_{L^2(\mathbb{R})}^2+\|b_x\|_{L^2(\mathbb{R})}^2\right)+\int_0^T\left(\mu\| u_{xx}\|_{L^2(\mathbb{R})}^2+\nu \|b_{xx}\|_{L^2(\mathbb{R})}^2\right)dt\leq C.$$
\end{lemma}
\begin{proof}
	Differentiating the equality $\eqref{a2}_3$ with respect to $x$,  then multiplying the resulting equation by $b_x$ and integrating by parts over $R$, we have
	\begin{equation}\label{c58}
	\begin{aligned}
	\frac{1}{2}\frac{d}{dt}\int b_x^2+\nu\int b_{xx}^2&=-2\int b_x^2 u_x-\int b b_x u_{xx}-\int u b_x b_{xx}\\
	&=-\frac{3}{2}\int b_x^2 u_x-\int b b_x u_{xx}.
	\end{aligned}
	\end{equation}
To control the second term on the right-hand side of \eqref{c58}, multiplying the momentum equation $\eqref{a2}_2$ by $u_{xx}$ gives
\begin{equation}\label{c59}
\mu\int u_{xx}^2=\int \rho \dot{u}u_{xx}+\gamma\int \rho^{\gamma-1}\rho_x u_{xx}+\int bb_x u_{xx}.
\end{equation}
Combining \eqref{c58} and \eqref{c59} yields
\begin{equation}\label{c60}
\begin{aligned}
\frac{1}{2}&\frac{d}{dt}\int b_x^2+\nu\int b_{xx}^2+\mu\int u_{xx}^2=-\frac{3}{2}\int b_x^2 u_x+\int \rho \dot{u} u_{xx}+\gamma\int \rho^{\gamma-1}\rho_x u_{xx}\\
&=-\frac{3}{2}\int b_x^2 \frac{F+P(\rho)-P(\bar{\rho})+\frac{b^2-\bar{b}^2}{2}}{\mu}+\int \rho \dot{u} u_{xx}+\gamma\int \rho^{\gamma-1}\rho_x u_{xx}\\
&\leq -\frac{3}{2}\int b_x^2 \frac{F-P(\bar{\rho})+\frac{-\bar{b}^2}{2}}{\mu}+\int \rho \dot{u} u_{xx}+\gamma\int \rho^{\gamma-1}\rho_x u_{xx}\\
&\leq C\left(\|F\|_{L^\infty}\|b_x\|_{L^2}^2+\|b_x\|_{L^2}^2+\|\sqrt{\rho}\dot{u}\|_{L^2}\|u_{xx}\|_{L^2}+\|\rho_x\|_{L^2}\|u_{xx}\|_{L^2}\right)\\
&\leq \varepsilon\|u_{xx}\|_{L^2}^2+C\left(\|F\|_{L^\infty}+1\right)\|b_x\|_{L^2}^2+C\left(\|\rho_x\|_{L^2}^2+\|\sqrt{\rho}\dot{u}\|_{L^2}^2\right)\\
&\leq\varepsilon\|u_{xx}\|_{L^2}^2+C\left(\|\sqrt{\rho}\dot{u}\|_{L^2}^2+\|b-\bar{b}\|_{L^6}^{6}+1\right)\|b_x\|_{L^2}^2+C\left(\|\rho_x\|_{L^2}^2+\|\sqrt{\rho}\dot{u}\|_{L^2}^2\right),
\end{aligned}
\end{equation}	
where we have used \eqref{c49} and Lemma \ref{lemmac4}.	

Next, differentiating the density equation $ \eqref{a2}_1$ with respect to $x$, multiplying the resulting equation by $\rho_x$, and integrating it over $\mathbb{R}$ implies that
\begin{equation}\label{c61}
\begin{aligned}
\frac{1}{2}\frac{d}{dt}\int \rho_x^2&=-2\int \rho_x^2 u_x-\int u\rho_x\rho_{xx}-\int \rho \rho_x u_{xx}\\
&=-\frac{3}{2}\int \rho_x^2 u_x-\int \rho \rho_x u_{xx}\\
&=-\frac{3}{2}\int \rho_x^2\frac{F+P(\rho)-P(\bar{\rho})+\frac{b^2-\bar{b}^2}{2}}{\mu}-\int \rho \rho_x u_{xx}\\
&\leq -\frac{3}{2\mu}\int \rho_x^2 F+\frac{3}{2\mu}\int \left(P(\bar{\rho})+\frac{\bar{b}^2}{2}\right)\rho_x^2+C\|\rho_x\|_{L^2}\|u_{xx}\|_{L^2}\\
&\leq \varepsilon\|u_{xx}\|_{L^2}^2+ C\left(\|F\|_{L^\infty}+1\right)\|\rho_x\|_{L^2}^2\\
&\leq  \varepsilon\|u_{xx}\|_{L^2}^2+ C\left(\|\sqrt{\rho}\dot{u}\|_{L^2}^2+\|b-\bar{b}\|_{L^6}^{6}+1\right)\|\rho_x\|_{L^2}^2,
\end{aligned}
\end{equation}	
where we have used \eqref{c49} and Lemma \ref{lemmac4}.	
Then adding the \eqref{c61} into \eqref{c60}, we have
	\begin{equation*}
	\begin{aligned}
	\frac{1}{2}\frac{d}{dt}\int \left(b_x^2+\rho_x^2\right)+\nu\int b_{xx}^2+\mu\int u_{xx}^2\leq C\left(\|\sqrt{\rho}\dot{u}\|_{L^2}^2+\|b-\bar{b}\|_{L^6}^{6}+1\right)\left(1+\|\rho_x\|_{L^2}^2+\|b_x\|_{L^2}^2\right),
	\end{aligned}
	\end{equation*}
which together Gronwall's inequality and Lemma \ref{lemmac4} gives
	\begin{equation*}\label{c62}
	\int\left( \rho_x^2+b_x^2\right)+\int_0^T\int \left(\nu b_{xx}^2+\mu u_{xx}^2\right)\leq C(T).	
	\end{equation*}
Thus, it follows from Lemma \ref{lemmac1} and the G-N inequality \eqref{b1-1} that
$$\|b\|_{L^\infty}\leq\|b-\bar{b}\|_{L^\infty}+\bar{b}\leq C\|b-\bar{b}\|_{L^2}^{\frac{1}{2}}\|b_x\|_{L^2}^{\frac{1}{2}}+\bar{b}\leq C(T).$$		
Thus, the proof of Lemma \ref{lemmac5} is finished.
\end{proof}	

With the help of the Lemma \ref{lemmac5}, we can get the estimates of the first order derivative with respect to $t$ of the density and the magnetic field, respectively.
	\begin{lemma}\label{lemmac6}
	Let $(\rho,u,b)$  be a smooth solution to \eqref{a2}-\eqref{a2-1}.
	Then for any $T>0$, it holds that	
		$$\sup_{0\leq t\leq T}\|\rho_t\|_{L^2(\mathbb{R})}+\|b_t\|_{L^2(0,T; L^2(\mathbb{R}))}\leq C.$$
		\end{lemma}
	\begin{proof}
		By direct calculation, we have
		\begin{equation*}\label{c63}
		\begin{aligned}
		\|\rho_t\|_{L^2}&\leq C\left(\|\rho u_x\|_{L^2}+\|u\rho_x\|_{L^2}\right)\\
		&\leq C\left(\|\rho\|_{L^\infty}\|u_x\|_{L^2}+\|u\|_{L^\infty}\|\rho_x\|_{L^2}\right)\in L^\infty(0,T),
		\end{aligned}
		\end{equation*}
		and
		\begin{equation*}\label{c64}
		\begin{aligned}
		\|b_t\|_{L^2}&\leq C\left(\|bu_x\|_{L^2}+\|ub_x\|_{L^2}+\|\nu b_{xx}\|_{L^2}\right)\\
		&\leq C\left(\|b\|_{L^\infty}\|u_x\|_{L^2}+\|u\|_{L^\infty}\|b_x\|_{L^2}+\|\nu b_{xx}\|_{L^2}\right)\\
		&\leq C\left(1+\|\nu b_{xx}\|_{L^2}\right)\in L^2(0,T),
		\end{aligned}
		\end{equation*}
		where we have used Lemma \ref{lemmac4} and Lemma \ref{lemmac5}.
	\end{proof}

The following estimates of the second order derivative of the velocity plays an important role in the analysis of the non-resistive limit.
\begin{lemma}\label{lemmac7}
	Let $(\rho,u,b)$  be a smooth solution to \eqref{a2}-\eqref{a2-1}.
	Then for any $T>0$, it holds that
	$$\sup_{0\leq t\leq T}\|\sqrt{\rho} \dot{u}\|_{L^2(\mathbb{R})}+\int_0^T\mu\| u_{xt}\|_{L^2(\mathbb{R})}^2\leq C,$$
and
$$\sup_{0\leq t\leq T}\left(\|u_{xx}\|_{L^2(\mathbb{R})}+\| \sqrt{\rho}u_{t}\|_{L^2(\mathbb{R})}\right)\leq C.$$
	\end{lemma}
\begin{proof}
	Differentiating the momentum equation $\eqref{a2}_2$ with respect to $t$, we have
	\begin{equation*}\label{c65}
	\rho _t\dot{u}+\rho \dot{u}_t+\left(P(\rho)+\frac{b^2}{2}\right)_{xt}=\mu u_{xxt}.
	\end{equation*}
Multiplying the above equation by $\dot{u}$ and integrating the resulting equation over $\mathbb{R}$ yields
	\begin{equation}\label{c66}
	\begin{aligned}
	\frac{1}{2}&\frac{d}{dt}\int \rho \dot{u}^2+\mu \int u_{xt}^2\\
	&=\frac{1}{2}\int \rho_t\dot{u}^2-\int \left(P(\rho)+\frac{b^2}{2}\right)_{xt}\dot{u}-\mu\int u_{xt}\left(u_x^2+uu_{xx}\right)\\
	&=:L_1+L_2+L_3.
	\end{aligned}
	\end{equation}
	Now, we estimate the terms $L_1-L_3$ as
	\begin{align}\label{c67}
	L_1&=-\frac{1}{2}\int (\rho u)_x\dot{u}^2=\int \rho u \dot{u}\dot{u}_x\notag\\
	&\leq C\|\sqrt{\rho}\dot{u}\|_{L^2}\|u\|_{L^\infty}\|\dot{u}_x\|_{L^2}\notag\\
	&\leq C\|\sqrt{\rho}\dot{u}\|_{L^2}\left(\|u_{xt}\|_{L^2}+\|u_x\|_{L^4}^2+\|u\|_{L^\infty}\|u_{xx}\|_{L^2}\right)\\
	&\leq C\|\sqrt{\rho}\dot{u}\|_{L^2}\left(\|u_{xt}\|_{L^2}+\|u_x\|_{L^2}^{\frac{3}{2}}\|u_{xx}\|_{L^2}^{\frac{1}{2}}+\|u_{xx}\|_{L^2}\right)\notag\\
	&\leq \varepsilon\|u_{xt}\|_{L^2}^2+C\left(\|\sqrt{\rho}\dot{u}\|_{L^2}^2+\|u_{xx}\|_{L^2}^2+1\right),\notag
	\end{align}
and	similarly, we also have
\begin{equation}\label{c68}
\begin{aligned}
L_2&=\int \left(P(\rho)+\frac{b^2}{2}\right)_t\dot{u}_x\\
&\leq C\|\rho^{\gamma-1}\rho_t+bb_t\|_{L^2}\|\dot{u}_x\|_{L^2}\\
&\leq C\left(\|\rho_t\|_{L^2}+\|b_t\|_{L^2}\right)\|u_{xt}+u_x^2+uu_{xx}\|_{L^2}\\
&\leq C\left(1+\|b_t\|_{L^2}\right)\left(\|u_{xt}\|_{L^2}+\|u_x\|_{L^2}^{\frac{3}{2}}\|u_{xx}\|_{L^2}^{\frac{1}{2}}+\|u\|_{L^\infty}\|u_{xx}\|_{L^2}\right)\\
&\leq \varepsilon\|u_{xt}\|_{L^2}^2+C\left(\|u_{xx}\|_{L^2}^2+\|b_t\|_{L^2}^2+1\right),
\end{aligned}
\end{equation}	
and
\begin{equation}\label{c69}
\begin{aligned}
L_3&\leq C\|u_{xt}\|_{L^2}\left(\|u_x\|_{L^4}^2+\|u\|_{L^\infty}\|u_{xx}\|_{L^2}\right)\\
&\leq \varepsilon\|u_{xt}\|_{L^2}^2+C\left(1+\|u_{xx}\|_{L^2}^2\right),
\end{aligned}
\end{equation}
where we have used the Gagliardo-Nirenberg inequality, Lemma \ref{lemmac4} and Lemma \ref{lemmac5}.
Thus, substituting \eqref{c67}-\eqref{c69} into \eqref{c66} and choosing $\varepsilon>0$ small enough, we have
\begin{equation*}
\frac{1}{2}\frac{d}{dt}\int \rho \dot{u}^2+\mu \int u_{xt}^2\leq C\left(1+\|\sqrt{\rho}\dot{u}\|_{L^2}^2+\|u_{xx}\|_{L^2}^2+\|b_t\|_{L^2}^2\right),
\end{equation*}
which combining with  Gronwall's inequality, Lemma \ref{lemmac4} and Lemma \ref{lemmac5} gives
$$\frac{1}{2}\int \rho \dot{u}^2+\int_0^T\int \mu u_{xt}^2\leq C(T).$$
This together with the momentum equation yields that
\begin{equation*}\label{c70}
\begin{aligned}
\mu \|u_{xx}\|_{L^2}&\leq C\left(\|\rho\dot{u}\|_{L^2}+\|\left(P(\rho)+\frac{b^2}{2}\right)_x\|_{L^2}\right)\\
&\leq C\left(\|\sqrt{\rho}\dot{u}\|_{L^2}+\|\rho_x\|_{L^2}+\|b_x\|_{L^2}\right)\leq C,
\end{aligned}
\end{equation*}
and
\begin{equation*}
\begin{aligned}
\|\sqrt{\rho}u_t\|_{L^2}&\leq C\|\sqrt{\rho}\dot{u}-\sqrt{\rho}uu_{x}\|_{L^2}\\
&\leq C\|\sqrt{\rho}\dot{u}\|_{L^2}+\|\sqrt{\rho}\|_{L^\infty}\|u\|_{L^\infty}\|u_{x}\|_{L^2}\leq C.
\end{aligned}
\end{equation*}
We complete the proof of Lemma \ref{lemmac7}.
\end{proof}

\section{Proof of the Theorem \ref{theorem1.1} }\label{s4}
This section is devoted to the proof of Theorem \ref{theorem1.1}. To do this, we also need the following global
a priori estimates of the solution $(\tilde{\rho},\tilde{u},\tilde{b})$ to the problem \eqref{a1} and \eqref{a1-1}.
\begin{proposition}\label{p4.1}
	Let  $(\tilde{\rho},\tilde{u},\tilde{b})$ be a smooth solution of \eqref{a1}, \eqref{a1-1} and \eqref{1.1}, then for some index $1<\alpha\leq 2$, we have
	 \begin{equation}\label{d1}
	0\leq \tilde{\rho} (x,t)\leq C,\ \ \ \forall (x,t)\in \mathbb{R}\times [0,T],\\
	\end{equation}
	\begin{equation}\label{d2}
	\begin{aligned}
	\sup_{0\leq t\leq T}&\left\|\left(\frac{1}{2}{\tilde{\rho}} \tilde{u}^2+\Phi(\tilde{\rho})+\frac{(\tilde{b}-\bar{b})^2}{2}\right)(1+|x|^\alpha)\right\|_{L^1} +
\int_0^T \mu\left\| \tilde{u}_x(1+|x|^{\frac{\alpha}{2}})\right\|_{L^2}^2dt\leq C,
	\end{aligned}
	\end{equation}
	and
	\begin{equation}\label{d3}
	\begin{aligned}
	\sup_{0\leq t\leq T}&\left(\|\tilde{u}\|_{H^2}+\|\tilde{b}_x\|_{L^2}+\|\tilde{\rho}_x\|_{L^2}+\|\sqrt{\tilde{\rho}}\dot{\tilde{u}}\|_{L^2}\right)+\int_0^T\mu\|\tilde{u}_{xt}\|_{L^2}^2\leq C.
\end{aligned}
\end{equation}
\end{proposition}
\begin{proof}
In fact, repeating the arguments in the proofs of Lemma \ref{lemmac1}-Lemma \ref{lemmac7} step by step, we can easily obtain \eqref{d1}- \eqref{d3}. Here, we omit the details.
\end{proof}

The existence and uniqueness of global strong solution in  Theorem \ref{theorem1.1} can be obtained from the local existence in time and the global (in time) a priori estimates in Proposition \ref{p4.1} by a standard continuity argument c.f. \cite{LWY2020} .

\section{Proof of the  Theorem \ref{theorem1.2} }\label{s5}
\textbf{\emph{ The existence and uniqueness of global strong solution in Theorem \ref{theorem1.2}:}}  The existence and uniqueness of global strong solution in  Theorem \ref{theorem1.2} can be obtained by the local existence in time and the global (in time) a priori estimates in Sections \ref{s3} by a standard continuity argument c.f. \cite{YL2019}.

\textbf{\emph{ The non-resistive limit in Theorem \ref{theorem1.2}:}} To justify the non-resistive limit as $\nu\rightarrow 0$, we consider the the difference of these
two solutions $(\rho-\tilde{\rho},u-\tilde{u},b-\tilde{b})$ which satisfy the following:
\begin{equation}\label{e1}
\begin{cases}
(\rho-\tilde{\rho})_t+(\rho-\tilde{\rho})u_x+\tilde{\rho}(u-\tilde{u})_x+(\rho-\tilde{\rho})_xu+\tilde{\rho}_x(u-\tilde{u})=0,\\
\rho(u-\tilde{u})_t+\rho u(u-\tilde{u})_x-\mu (u-\tilde{u})_{xx}=-(\rho-\tilde{\rho})(\tilde{u}_t+\tilde{u}\tilde{u}_x)-\rho(u-\tilde{u})\tilde{u}_x\\
\qquad\qquad\qquad-\left(P(\rho)-P(\tilde{\rho})\right)_x-\left(\frac{b^2-\tilde{b}^2}{2}\right)_x,\\
(b-\tilde{b})_t+u_x(b-\tilde{b})+\tilde{b}(u-\tilde{u})_x+u(b-\tilde{b})_x+(u-\tilde{u})\tilde{b}_x=\nu b_{xx}.
\end{cases}
\end{equation}

Firstly, multiplying the equation $\eqref{e1}_1$ by $2(\rho-\tilde{\rho})$, integrating the resultant over $\mathbb{R}$, and integrating by
parts, it follows from  H\"{o}lder's inequality, the Gagliardo-Nirenberg inequality, the Cauchy inequality, Lemma \ref{lemmac4} and Proposition \ref{p4.1} that
\begin{equation}\label{e2}
\begin{aligned}
&\frac{d}{dt}\int (\rho -\tilde{\rho})^2=-\int (\rho -\tilde{\rho})^2 u_x-2\int \tilde{\rho}(\rho -\tilde{\rho})(u-\tilde{u})_x-2\int \tilde{\rho}_x(\rho -\tilde{\rho})(u-\tilde{u})\\
&\leq \big(\|u_x\|_{L^\infty}\|\rho -\tilde{\rho}\|_{L^2}^2+\|\tilde{\rho}\|_{L^\infty}\|\rho -\tilde{\rho}\|_{L^2}\|(u-\tilde{u})_x\|_{L^2}+\|\tilde{\rho}_x\|_{L^2}\|\rho-\tilde{\rho}\|_{L^2}\|u-\tilde{u}\|_{L^\infty}\big)\\
&\leq C\left[\|u_x\|_{L^2}^{\frac{1}{2}}\|u_{xx}\|_{L^2}^{\frac{1}{2}}\|\rho-\tilde{\rho}\|_{L^2}^2+\left(\|(u-\tilde{u})_x\|_{L^2}+\|u-\tilde{u}\|_{L^2}^{\frac{1}{2}}\|(u-\tilde{u})_x\|_{L^2}^{\frac{1}{2}}\right)
\|\rho-\tilde{\rho}\|_{L^2}\right]\\
&\leq \varepsilon\|(u-\tilde{u})_x\|_{L^2}^2+C\left(1+\|u_{xx}\|_{L^2}\right)\left(\|\rho-\tilde{\rho}\|^2_{L^2}+1\right).
\end{aligned}
\end{equation}

Secondly, we multiply  $\eqref{e1}_2$ by $2(u-\tilde{u})$,  integrate the resultant over $\mathbb{R}$, integrate by
parts,  and use H\"{o}lder's inequality, the Gagliardo-Nirenberg inequality, Young's inequality,
the fact that $\|u-\tilde{u}\|^2_{L^2}\le C\left(1+\|(u-\tilde{u})_x\|_{L^2}^2\right)$ (see \eqref{c12-1}) and Proposition \ref{p4.1} to get
\begin{equation}\label{e3}
\begin{aligned}
&\frac{d}{dt}\int \rho (u-\tilde{u})^2+2\mu \int (u-\tilde{u})_x^2=-\int 2(\rho-\tilde{\rho})(u-\tilde{u})(\tilde{u}_t+\tilde{u}\tilde{u}_x)\\
&\quad-\int 2\rho (u-\tilde{u})^2\tilde{u}_x-\int2 \left(P(\rho)-P(\tilde{\rho})\right)_x(u-\tilde{u})-\int (b^2-\tilde{b}^2)_x(u-\tilde{u})\\
&\leq C\Big(\|\dot{\tilde{u}}\|_{L^2}\|\rho -\tilde{\rho}\|_{L^2}\|u-\tilde{u}\|_{L^\infty}+\|\tilde{u}_x\|_{L^\infty}\|\sqrt{\rho}(u-\tilde{u})\|_{L^2}^2\\
&\quad+(\|\rho-\tilde{\rho}\|_{L^2}+\|b-\tilde{b}\|_{L^2})\|(u-\tilde{u})_x\|_{L^2}\Big)\\
&\leq C\Big[\|\dot{\tilde{u}}\|_{L^2}\|u-\tilde{u}\|_{L^2}^{\frac{1}{2}}\|(u-\tilde{u})_x\|_{L^2}^{\frac{1}{2}}\|\rho-\tilde{\rho}\|_{L^2}+\|\tilde{u}_x\|_{L^2}^{\frac{1}{2}}\|\tilde{u}_{xx}\|_{L^2}^{\frac{1}{2}}\|\sqrt{\rho}(u-\tilde{u})\|_{L^2}^2\\
& \quad +(\|\rho-\tilde{\rho}\|_{L^2}+\|b-\tilde{b}\|_{L^2})\|(u-\tilde{u})_x\|_{L^2}\Big]\\
&\leq \varepsilon\|(u-\tilde{u})_x\|_{L^2}^2+C\left(1+\|\rho-\tilde{\rho}\|_{L^2}^2+\|b-\tilde{b}\|_{L^2}^2+\|\sqrt{\rho}(u-\tilde{u})\|_{L^2}^2\right)\\
&\quad\left(1+\|\dot{\tilde{u}}\|_{L^2}^2+\|\tilde{u}_{xx}\|_{L^2}\right).
\end{aligned}
\end{equation}

Thirdly, multiplying the equation $\eqref{e1}_3$ by $2(b-\tilde{b})$, integrating the resultant over $\mathbb{R}$, integrating by
parts and using Proposition \ref{p4.1}, the third equation of \eqref{a1}, the fact that $\|\tilde{b}\|_{L^\infty}\le C \|\tilde{b}\|^{1\over2}_{L^2}
\|\tilde{b}_x\|_{L^2}^{1\over2}$, we deduce
\begin{equation}\label{e4}
\begin{aligned}
&\frac{d}{dt}\int (b-\tilde{b})^2\\
&=-\int (b-\tilde{b})^2 u_x-2\int \tilde{b}(b-\tilde{b})(u-\tilde{u})_x-2\int (b-\tilde{b})(u-\tilde{u})\tilde{b}_x +2\nu \int (b-\tilde{b})b_{xx}\\
&\leq C\Big(\|u_x\|_{L^\infty}\|b-\tilde{b}\|_{L^2}^2+\|\tilde{b}\|_{L^\infty}\|b-\tilde{b}\|_{L^2}\|(u-\tilde{u})_x\|_{L^2}\\
&\quad+\|b-\tilde{b}\|_{L^2}\|\tilde{b}_x\|_{L^2}\|u-\tilde{u}\|_{L^\infty}+\nu\|b-\tilde{b}\|_{L^2}\|b_{xx}\|_{L^2}\Big)\\
&\le\varepsilon \|(u-\tilde{u})_x\|_{L^2}^2+C\left(1+\|u_{xx}\|_{L^2}\right)\|b-\tilde{b}\|_{L^2}^2+\nu^2\|b_{xx}\|_{L^2}^2.
\end{aligned}
\end{equation}
Then combining \eqref{e2}-\eqref{e4} and choosing $\varepsilon>0$ sufficiently small, we have
\begin{equation*}\label{e5}
\begin{aligned}
&\frac{d}{dt}\int \left((\rho-\tilde{\rho})^2+\rho (u-\tilde{u})^2+(b-\tilde{b})^2\right)+2\mu \int (u-\tilde{u})_x^2\\
&\leq C\left(1+\|u_{xx}\|_{L^2}+\|\tilde{u}_{xx}\|_{L^2}+\|\dot{\tilde{u}}\|_{L^2}^2\right)\left(\|\rho-\tilde{\rho}\|_{L^2}^2
+\|b-\tilde{b}\|_{L^2}^2+\|\sqrt{\rho}(u-\tilde{u})\|_{L^2}^2+1\right)\\
&\quad +\nu^2\|b_{xx}\|_{L^2}^2,
\end{aligned}
\end{equation*}
which together with  Gronwall's inequality, Lemma \ref{lemmac5} and Proposition \ref{p4.1} yields
\begin{equation}\label{e6}
\begin{aligned}
\int &\left((\rho-\tilde{\rho})^2+\rho (u-\tilde{u})^2+(b-\tilde{b})^2\right)+2\mu \int (u-\tilde{u})_x^2\\
&\leq C\nu^2\int_0^T\|b_{xx}\|_{L^2}^2dt\exp\left\{C\int_0^T\left(1+\|u_{xx}\|_{L^2}+\|\tilde{u}_{xx}\|_{L^2}+\|\dot{\tilde{u}}\|_{L^2}^2\right)dt\right\}\\
&\leq C\nu,
\end{aligned}
\end{equation}
where we have used Lemmas \ref{lemmac5} and \ref{lemmac7}, Proposition \ref{p4.1} and the following fact:
\begin{equation*}
\begin{aligned}
\|\dot{\tilde{u}}\|_{L^2}&\leq C\|\sqrt{\bar{\rho}}\dot{\tilde{u}}\|_{L^2}\leq C\left(\|(\sqrt{\bar{\rho}}-\sqrt{\tilde{\rho}})\tilde{u}\|_{L^2}+\|\sqrt{\tilde{\rho}}\tilde{u}\|_{L^2}\right)\\
&\leq C\left(\|\tilde{\rho}-\bar{\rho}\|_{L^2}\|\dot{\tilde{u}}\|_{L^\infty}+1\right)\\
&\leq \varepsilon\|\dot{\tilde{u}}\|_{L^2}+C(\|\dot{\tilde{u}}_x\|_{L^2}+1),\\
\end{aligned}
\end{equation*}
which gives
\begin{equation*}
\begin{aligned}
\|\dot{\tilde{u}}\|_{L^2}&\leq C(\|\left(\tilde{u}_t+\tilde{u}\tilde{u}_x\right)_x\|_{L^2}+1)\\
&\leq C\left(\|\tilde{u}_{xt}\|_{L^2}+\|\tilde{u}_x\|_{L^4}^2+\|\tilde{u}\tilde{u}_{xx}\|_{L^2}+1\right)\\
&\leq C\left(\|\tilde{u}_{xt}\|_{L^2}+\|\tilde{u}_{xx}\|_{L^2}+1\right)\in L^2(0,T).
\end{aligned}
\end{equation*}
Moreover, we have
\begin{equation*}
\begin{aligned}
\int \bar{\rho}(u-\tilde{u})^2&=\int (\rho-\bar{\rho})(u-\tilde{u})^2+\int \rho (u-\tilde{u})^2\\
&\leq C\|\rho-\bar{\rho}\|_{L^2}\|u-\tilde{u}\|_{L^4}^2+C\nu\\
&\leq C\|\rho-\bar{\rho}\|_{L^2}\|u-\tilde{u}\|_{L^2}^{\frac{3}{2}}\|(u-\tilde{u})_x\|_{L^2}^{\frac{1}{2}}+C\nu\\
&\leq \varepsilon\|u-\tilde{u}\|_{L^2}^2+C\|\rho-\bar{\rho}\|_{L^2}^4\|(u-\tilde{u})_x\|_{L^2}^2+C\nu\\
&\leq \varepsilon\|u-\tilde{u}\|_{L^2}^2+C\nu^2+C\nu
\end{aligned}
\end{equation*}
which together with \eqref{e6} yields that
\begin{equation}\label{e7}
\sup_{0\leq t\leq T}\|u-\tilde{u}\|_{L^2}^2\leq C\nu.
\end{equation}
Then we have proved the results of Theorem \ref{theorem1.2}.  \hspace{58mm}\qedsymbol

\section {Acknowledgments}
Z. Li is supported by the NSFC (No.11671319,
No.11931013), Fund of HPU (No.B2016-57).
H. Wang is supported by the National Natural Science Foundation of China (Grant No.11901066), the
Natural Science Foundation of Chongqing (Grant No.cstc2019jcyj-msxmX0167) and Project No.2019CDXYST0015 and No. 2020CDJQY-A040
supported by the Fundamental Research Funds for the Central Universities. Y. Ye is supported by the National Natural Science Foundation of China (No.11701145) and China Postdoctoral Science Foundation (No.2020M672196).

\end{document}